\documentclass[11pt]{amsart}

\usepackage{amsfonts}
\usepackage{amsmath}
\usepackage{amssymb}
\usepackage{amsthm}
\usepackage{amsxtra}
\usepackage{epic, eepic}
\usepackage{epsfig,color}
\usepackage{fullpage}
\usepackage{graphicx}
\usepackage{subfigure}
\usepackage{vmargin}
\theoremstyle{plain}

\newtheorem{Thm}{Theorem}[section]
\newtheorem{Cor}[Thm]{Corollary}
\newtheorem{Lem}[Thm]{Lemma}

\theoremstyle{definition}

\newtheorem*{Pb}{Problem}

\theoremstyle{remark}
\newtheorem{Rem}{Remark}[section]

\newtheorem{Ex}{Example}[section]

\numberwithin{equation}{section}

\setmarginsrb{33mm}{30mm}{33mm}{40mm}%
            {0mm}{10mm}{0mm}{10mm}

\begin{document}

\newcommand{\rx}{\mathcal{R}(X)}

% Title
\title{Graph colorings, flows\\ and\\ arithmetic Tutte polynomial}

\author{Michele D'Adderio$^*$}\thanks{$^*$ supported by the Max-Planck-Institut
f\"{u}r Mathematik and the University of
G\"{o}ttingen.}\author{Luca Moci$^{\dag}$}\thanks{$^{\dag}$
supported by Dipartimento di Matematica "Guido Castelnuovo"
(Roma).}
\address{Georg-August Universit\"{a}t G\"{o}ttingen\\ Mathematisches Institut\\
Bunsenstrasse 3-5, D-37073 G\"{o}ttingen\\
Germany}\email{mdadderio@yahoo.it}

\address{Dipartimento di Matematica "Guido Castelnuovo"\\Sapienza Universit\`a di Roma\\
Piazzale Aldo Moro 5, 00185 Roma\\
Italy}\email{moci@mat.uniroma1.it}

\begin{abstract}
We introduce the notions of arithmetic colorings and arithmetic
flows over a graph with labelled edges, which generalize the
notions of colorings and flows over a graph.

We show that the corresponding arithmetic chromatic polynomial and
arithmetic flow polynomial are given by suitable specializations
of the associated arithmetic Tutte polynomial, generalizing
classical results of Tutte \cite{Tu}.
\end{abstract}

\maketitle
%\tableofcontents

\section*{Introduction}

It is well known how to associate a matroid, and hence a Tutte
polynomial, to a graph. Moreover, several combinatorial objects
associated to a graph are counted by suitable specializations of
this important invariant: (proper) $q$-colorings and (nowhere
zero) $q$-flows are two of the most classical examples. See
\cite{GR} and \cite{Wh} for systematic accounts.

In \cite{MoT} a new polynomial has been introduced, which provides
a natural counterpart for toric arrangements of the Tutte
polynomial of a hyperplane arrangement. In fact in \cite{MoT} and
\cite{DM} it has been shown how this polynomial has several
applications to toric arrangements, vector partition functions and
zonotopes.

In \cite{DM2} we introduced the notion of an \textit{arithmetic
matroid}, which generalizes the one of a matroid, and whose main
example is provided by a list of elements in a finitely generated
abelian group. To this object we associated an \textit{arithmetic
Tutte polynomial} (which is the polynomial in \cite{MoT} for the
main example) and provided a combinatorial interpretation of it
which extends the one given by Crapo for the classical Tutte
polynomial.

Encouraged by all these evidences (cf. also \cite{li} where two
parallel theories are developed), we think of the arithmetic
matroids and the arithmetic Tutte polynomial as natural
generalizations of their classical counterparts. So it seemed also
natural to us to look for applications in graph theory.

In this paper we introduce the notion of an \textit{arithmetic
(proper) $q$-coloring} and an \textit{arithmetic (nowhere zero)
$q$-flow} over a graph $(\mathcal{G},\ell)$ with labelled (by
integers) edges, which generalize the well known notions of
$q$-colorings and $q$-flows of a graph. We then associate to the
labelled graph an arithmetic matroid
$\mathfrak{M}_{\mathcal{G},\ell}$, and we show how suitable
specializations of its arithmetic Tutte polynomial provides the
\textit{arithmetic chromatic polynomial}
$\chi_{\mathcal{G},\ell}(q)$ and the \textit{arithmetic flow
polynomial} $\chi_{\mathcal{G},\ell}^*(q)$ of the labelled graph
(see\textbf{ Theorem \ref{thm:maincolor}} and \textbf{Theorem
\ref{thm:mainflow}}).

These can be seen as generalizations of the classical results of
Tutte \cite{Tu} (\textbf{Corollaries \ref{cor:color}} and
\textbf{\ref{cor:flow}}) that the chromatic polynomial and the
flow polynomial of a graph can be obtained as suitable
specializations of the corresponding Tutte polynomial.

The paper is organized in the following way.

In the first section we define the basic notions of graph theory
that we need, in particular the notion of \textit{labelled graph},
and we fix the corresponding notation.

In the second section we recall some basic notions of the theory
of arithmetic matroids. In particular we state some of their basic
properties, and we show how to associate to a labelled graph an
arithmetic matroid.

In the third section we introduce the notion of arithmetic
coloring and we state Theorem \ref{thm:maincolor}.

In the fourth section we prove Theorem \ref{thm:maincolor}.

In the fifth section we introduce the notion of arithmetic flow
and we state Theorem \ref{thm:mainflow}.

In the sixth section we prove Theorem \ref{thm:mainflow}.

In the last section we make some final comments and we formulate
an open problem.

\subsection*{Acknowledgments}

We would like to thank Petter Br\"{a}nd\'{e}n and Matthias Lenz
for interesting discussions.

\section{Labelled graphs}

In this paper a \textit{graph} $\mathcal{G}$ will be a pair
$(V,E)$, where $V$ is a finite set whose elements are called
\textit{vertices}, and $E$ is a finite multiset of 2-element
multisets of $V$, which are called \textit{edges}. A \textit{loop}
is an edge whose elements coincide.

\bigskip

\textit{In this paper we will always assume that our graphs have
no loops.}

\bigskip

\begin{Ex}
Consider $\mathcal{G}:=(V,E)$, where $V:=\{v_1,v_2,v_3,v_4\}$ is
the set of vertices and $E:=\{\{v_1,v_2\}$, $\{v_2,v_3\}$,
$\{v_2,v_4\}$, $\{v_3,v_4\}\}$ is the set of edges (see Figure 1).
\begin{figure}[h]
\includegraphics[width=60mm,clip=true,trim=10mm 180mm 30mm 10mm]{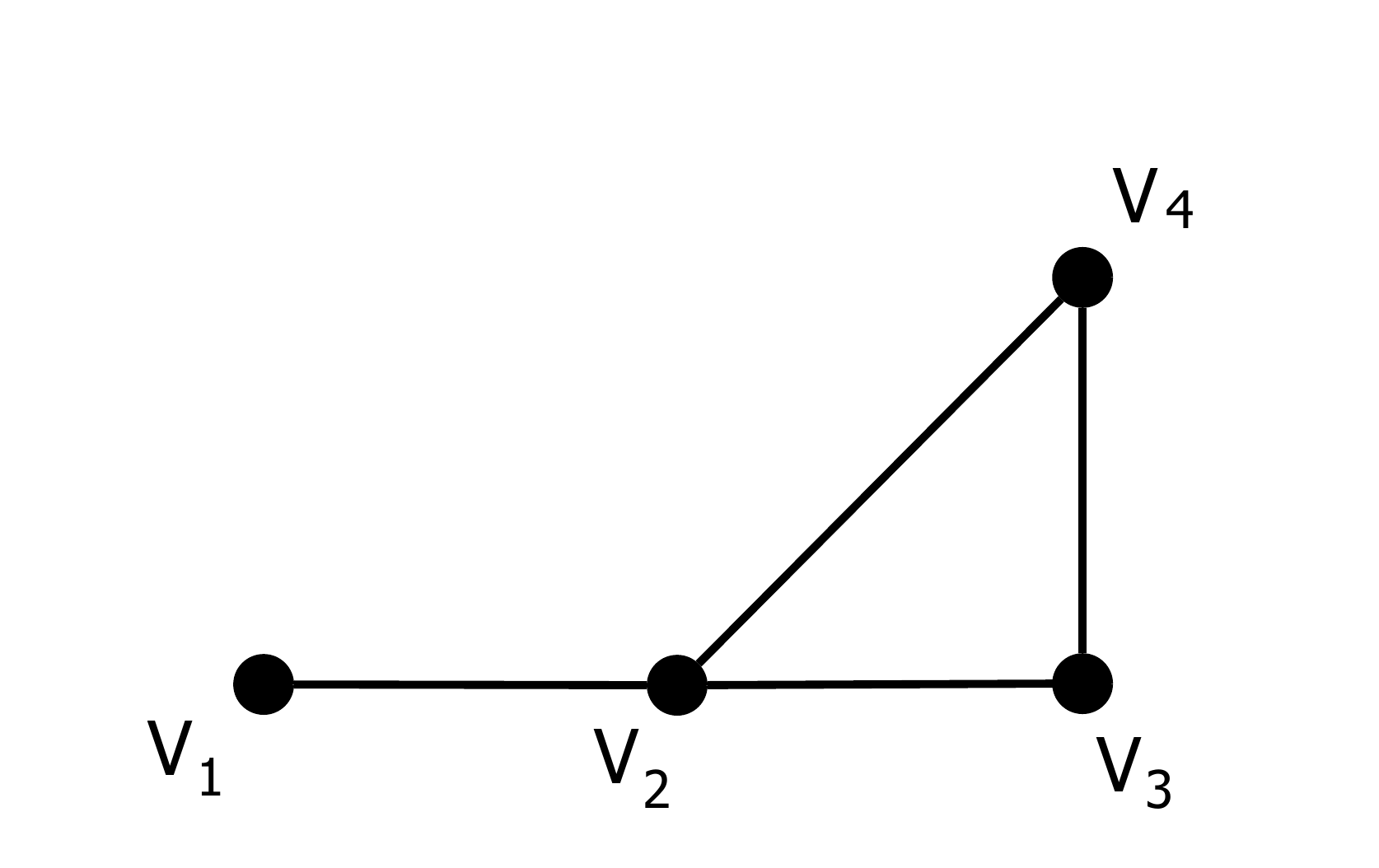}
\caption{The graph $\mathcal{G}$.}
\end{figure}
\end{Ex}

We define the \textit{classical deletion} of an edge $e$ of a
graph $\mathcal{G}=(V,E)$ to be simply the graph $(V,E\backslash
\{e\})$, i.e. the graph $\mathcal{G}$ with the edge $e$ removed.

We define the \textit{classical contraction} of an edge $e$ of a
graph $\mathcal{G}=(V,E)$ to be the graph $\mathcal{G}$ with the
edge $e$ removed and with the corresponding vertices identified.

\begin{Ex}
Let $\mathcal{G}:=(V,E)$, where $V:=\{v_1,v_2,v_3,v_4\}$ is the
set of vertices and $E:=\{\{v_1,v_2\}$, $\{v_2,v_3\}$,
$\{v_2,v_4\}$, $\{v_3,v_4\}\}$ is the set of edges. Let
$e:=\{v_2,v_3\}\in E$. Then the classical deletion of $e$ is the
graph with vertices $V$, and edges $E\setminus
e=\{\{v_1,v_2\},\{v_2,v_4\},\{v_3,v_4\}\}$, while the classical
contraction of $e$ is the graph with vertices
$V':=\{v_1',v_2',v_3'\}$ and edges $E':=\{\{v_1',v_2'\}$,
$\{v_2',v_3'\}$, $\{v_2',v_3'\}\}$ (see Figure 2).
\end{Ex}
\begin{figure}[h]
%$\begin{array}{cc}
\subfigure[Classical deletion of
$e$]{\includegraphics[width=60mm,clip=true,trim=10mm 180mm 30mm
10mm]{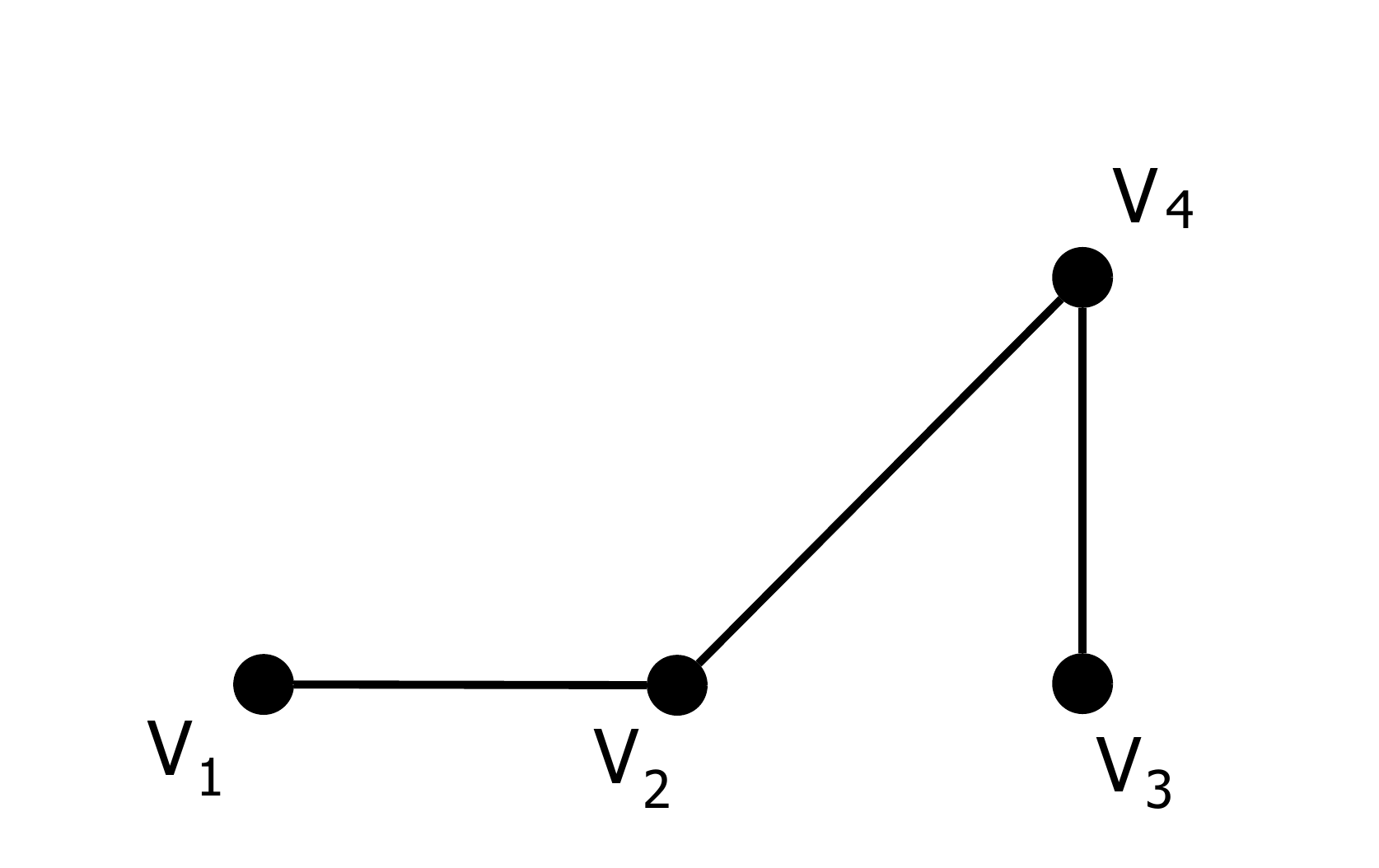}} \subfigure[Classical contraction of
$e$]{\includegraphics[width=60mm,clip=true,trim=10mm 180mm 30mm
10mm]{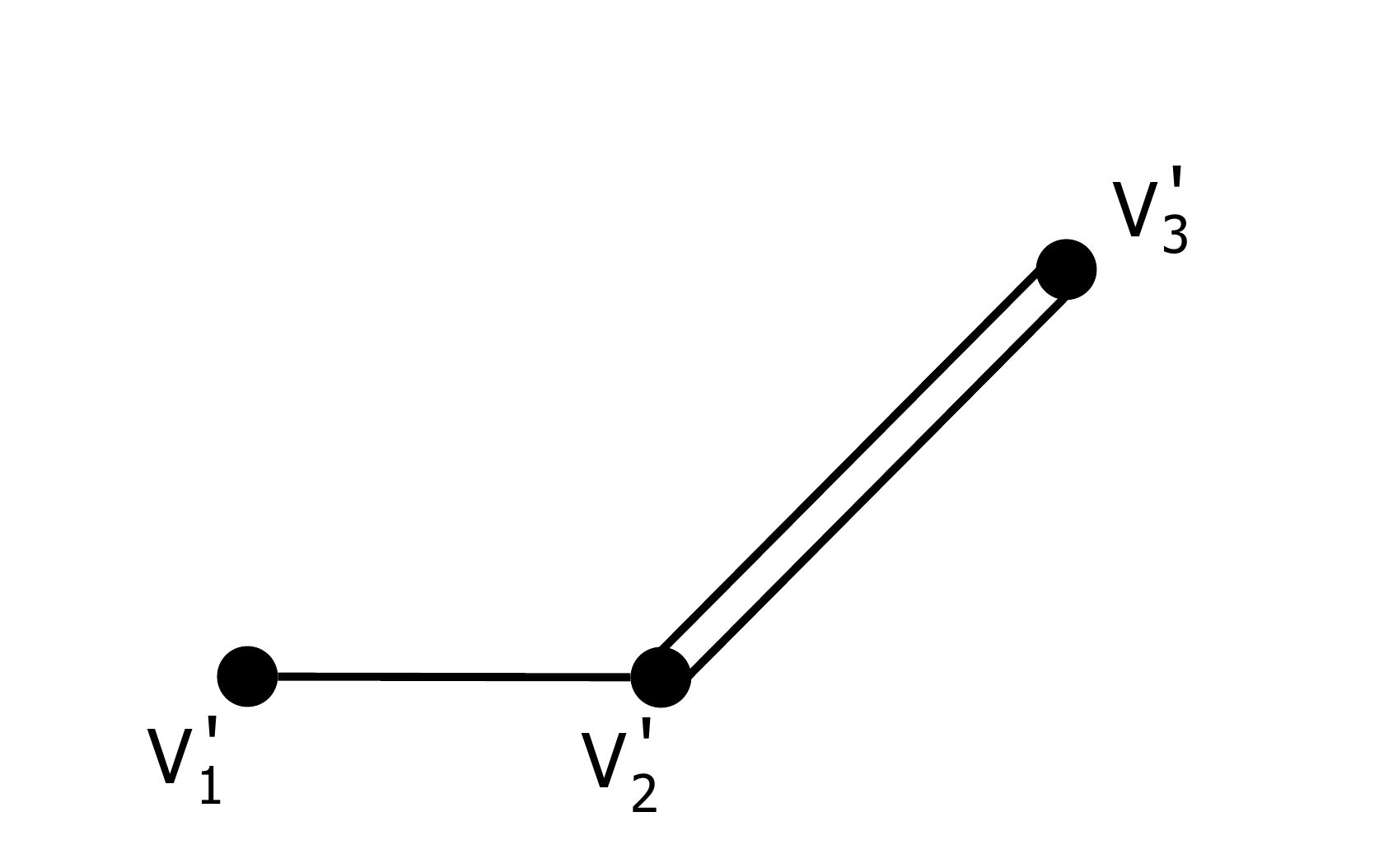}}
%\end{array}$

\caption{The classical deletion and contraction of $e$}
\end{figure}

We distinguish two kinds of edges: we assume that $E$ is a
disjoint union $E=R\cup D$, where we call the elements of $R$
\textit{regular} edges, while we call the elements of $D$
\textit{dotted} edges.

A \textit{labelled graph} in this contest will be simply a pair
$(\mathcal{G},\ell)$, where $\mathcal{G}=(V,E)$ is a graph, and
$\ell:E\rightarrow \mathbb{N}\setminus \{0\}$ is a map, whose
images are called \textit{labels} of the corresponding edges.

\begin{Ex} \label{ex:labelgraph}
Consider $(\mathcal{G},\ell)$, where $\mathcal{G}:=(V,E)$,
$V:=\{v_1,v_2,v_3,v_4\}$ is the set of vertices,
$R:=\{\{v_1,v_2\},\{v_2,v_3\},\{v_2,v_4\}\}$ the set of regular
edges, $D:=\{\{v_3,v_4\}\}$ the set of dotted edges, so that
$E=R\cup D=\{\{v_1,v_2\},\{v_2,v_3\},\{v_2,v_4\},\{v_3,v_4\}\}$.
Moreover let $\ell(\{v_1,v_2\})=1$, $\ell(\{v_2,v_3\})=2$,
$\ell(\{v_2,v_4\})=3$, $\ell(\{v_3,v_4\})=6$ be the labels of the
edges (see Figure 3).

\begin{figure}[h]
\includegraphics[width=60mm,clip=true,trim=10mm 180mm 30mm 10mm]{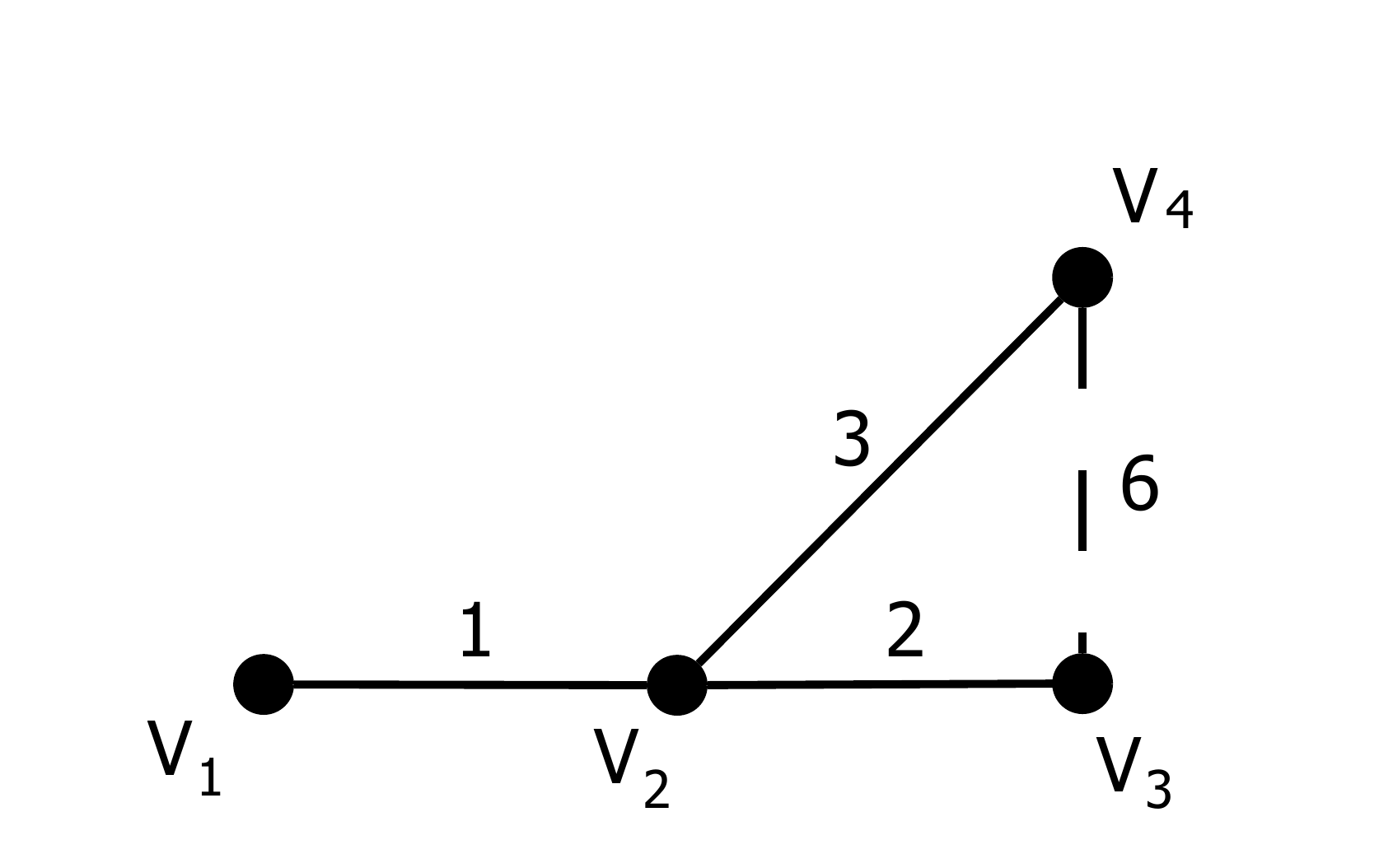}
\caption{The labelled graph $(\mathcal{G},\ell)$.}
\end{figure}
\end{Ex}

A \textit{directed graph} is a pair $(V,E)$ where $V$ is a finite
set of vertices, and $E$ is a finite multiset of ordered pairs of
elements of $V$ that we call \textit{directed edges}. For a
directed edge $e\in E$ we will denote by $e^+$ and $e^-$ the first
and the second coordinate of $e$ respectively. Pictorially, to
denote a directed edge we draw an arrow pointing toward its first
coordinate.

Given a graph $\mathcal{G}=(V,E)$, an \textit{orientation}
$E_{\theta}=R_{\theta}\cup D_{\theta}$ of the edges $E=R\cup D$ is
a multiset of ordered pairs of elements of $V$ whose underlying
sets are the elements of $E$. We will call $\mathcal{G}_{\theta}$
the corresponding directed graph.

Given a labelled graph $(\mathcal{G},\ell)$, where
$\mathcal{G}=(V,E)$ and $E=R\cup D$, we define the
\textit{deletion} of a regular edge $e \in R$ to be the pair
$(\mathcal{G}-e,\ell_1)$, where $\mathcal{G}-e=(V,E_1)$ is the
classical deletion of the edge $e$ (i.e. $E_1=R_1\cup D_1$,
$R_1=R\setminus \{e\}$, $D_1=D$), and $\ell_1$ is simply the
restriction of $\ell$ to $E_1$; we define the \textit{contraction}
of $e\in R$ to be the pair $(\mathcal{G}/e,\ell_2)$, where
$\mathcal{G}/e=(V,E_2)$ is the graph obtained from $\mathcal{G}$
by removing $e$ from $R$ and putting it in $D$, i.e. making the
regular edge $e$ into a dotted one (i.e $E_2=R_2\cup D_2$,
$R_2=R\setminus \{e\}$, $D_2=D\cup \{e\}$), and $\ell_2$ is the
same as $\ell$.

\begin{Ex}
Consider $(\mathcal{G},\ell)$ as in Example \ref{ex:labelgraph}.
For $e:=\{v_2,v_3\}\in R$ we show the deletion and the contraction
of $e$ in Figure 4.

\begin{figure}[h]
%$\begin{array}{cc}
\subfigure[Deletion of
$e$]{\includegraphics[width=60mm,clip=true,trim=10mm 180mm 30mm
10mm]{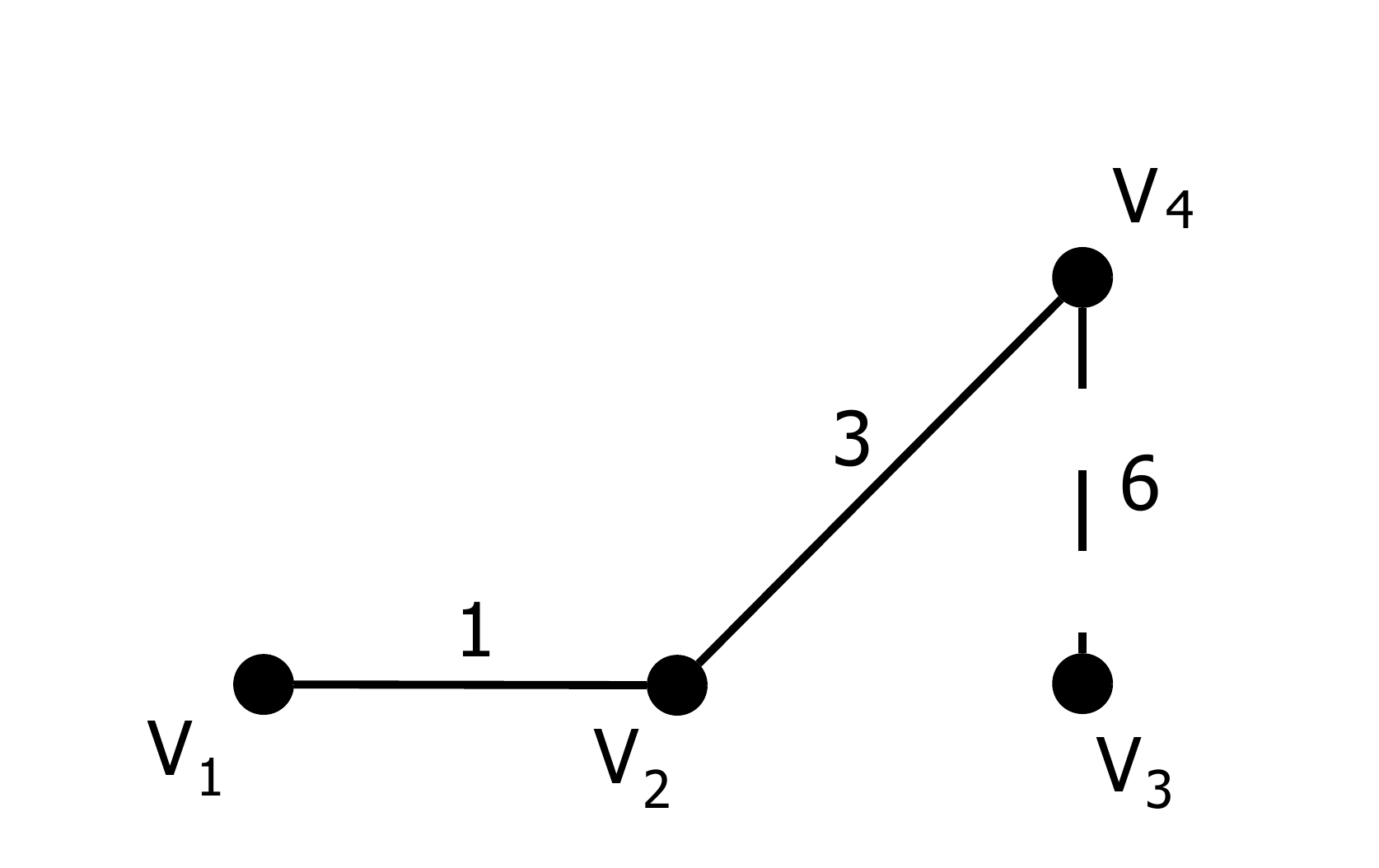}} \subfigure[Contraction of
$e$]{\includegraphics[width=60mm,clip=true,trim=10mm 180mm 30mm
10mm]{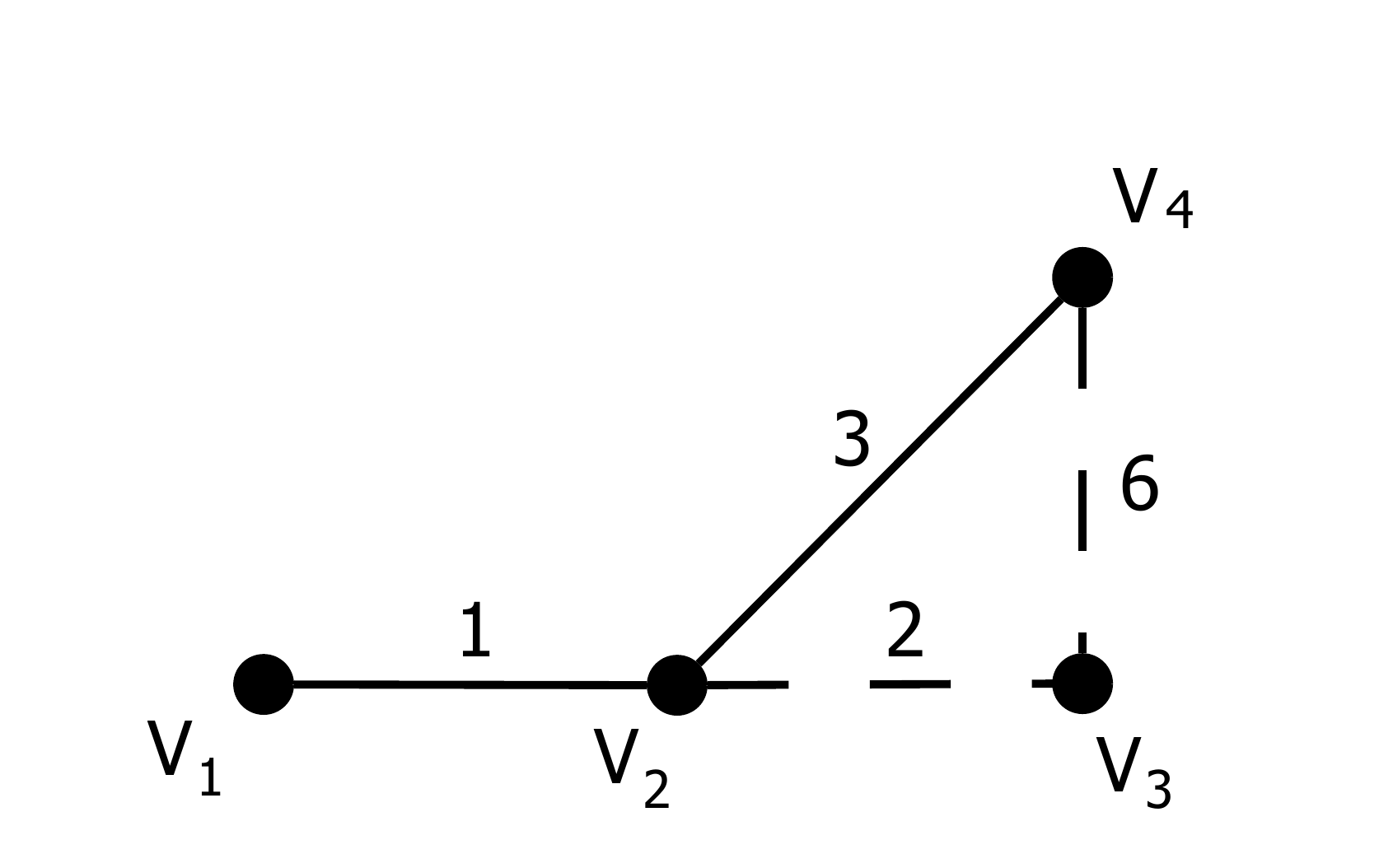}}
%\end{array}$

\caption{The deletion and the contraction of $e$}
\end{figure}
\end{Ex}

Given a graph $\mathcal{G}=(V,E)$ with $E=R\cup D$, we denote by
$\overline{\mathcal{G}}=(\overline{V},\overline{E})$ the graph
obtained from $\mathcal{G}$ by (classically) contracting the edges
in $D$. For a regular edge $e\in R$, we will use the notation
$\overline{\mathcal{G}- e}=(\overline{V}_1,\overline{E}_1)$ and
$\overline{\mathcal{G}/ e}=(\overline{V}_2,\overline{E}_2)$.

\section{Arithmetic matroids}

We recall here same basic notions of the theory of arithmetic
matroids. We refer to \cite{DM2} for proofs and for a more
systematic treatment.

We will use the word \emph{list} as a synonymous of multiset.
Hence a list may contain several copies of the same element.

\subsection{Matroids}

A \textit{matroid} $\mathfrak{M}=\mathfrak{M}_X=(X,rk)$ is a list
of \textit{vectors} $X$ with a \textit{rank function}
$rk:\mathbb{P}(X)\rightarrow \mathbb{N}\cup \{0\}$ which satisfies
the following axioms:
\begin{enumerate}
    \item if $A\subseteq X$, then $rk(A)\leq |A|$;
    \item if $A,B\subseteq X$ and $A\subseteq B$, then $rk(A)\leq rk(B)$;
    \item if $A,B\subseteq X$, then $rk(A\cup B)+rk(A\cap B)\leq
    rk(A)+rk(B)$.
\end{enumerate}

A sublist $A$ of $X$ is called \textit{independent} if
$rk(A)=|A|$. It is easy to show that the independent sublists
determine the matroid structure.

\begin{Ex}
\begin{enumerate}
    \item A list $X$ of vectors in a vector space, where the independent
sublists are defined to be the linearly independent ones naturally
form a matroid.
    \item A list $X$ of edges in a graph, where the independent
    sublists are the edges of the subgraphs that are forests (i.e. the subgraphs without circuits)
    naturally form a matroid.
\end{enumerate}
\end{Ex}

Given a matroid $\mathfrak{M}_X$ and a vector $v\in X$, we can
define the \textit{deletion} of $\mathfrak{M}_X$ as the matroid
$\mathfrak{M}_{X_1}$, whose list of vectors is $X_1:=X\setminus
\{v\}$, and whose independent lists are just the independent lists
of $\mathfrak{M}_X$ contained in $X_1$. Notice that the rank
function $rk_1$ of $\mathfrak{M}_{X_1}$ is just the restriction of
the rank function $rk$ of $\mathfrak{M}_{X}$.

Given a matroid $\mathfrak{M}_X$ and a vector $v\in X$, we can
define the \textit{contraction} of $\mathfrak{M}_X$ as the matroid
$\mathfrak{M}_{X_2}$, whose list of vectors is $X_2:=X\setminus
\{v\}$, and whose rank function $rk_2$ is given by
$rk_2(A):=rk(A\cup \{v\})-rk(\{v\})$, where of course $rk$ is the
rank function of $\mathfrak{M}_X$.

\begin{Ex}
\begin{enumerate}
    \item For a matroid given by a list $X$ of vectors in a vector space,
    the deletion consists of removing the vector from the list,
    while the contraction consists of removing it from the list,
    taking the quotient by the subspace that it generates, and
    identifying the remaining vectors with their cosets.
    \item For a matroid given by a list $X$ of edges in a graph, the deletion
    consists of removing the edge from the list, i.e. the classical deletion of the edge,
    while the contraction corresponds to the classical contraction
    of the edge.
\end{enumerate}
\end{Ex}

Given two matroids $\mathfrak{M}_{X_1}=(X_1,I_1)$ and
$\mathfrak{M}_{X_2}=(X_2,I_2)$, we can form their \textit{direct
sum}: this will be the matroid
$\mathfrak{M}_X=\mathfrak{M}_{X_1}\oplus \mathfrak{M}_{X_2}$ whose
list of vectors is the disjoint union $X:=X_1\sqcup X_2$, and
where the independent lists will be the disjoint unions of lists
from $I_1$ with lists from $I_2$. Hence for any sublist
$A\subseteq X$, the rank $rk(A)$ of $A$ will be the sum of the
rank $rk_1(A\cap X_1)$ of $A\cap X_1$ in $\mathfrak{M}_{X_1}$ with
the rank $rk_2(A\cap X_2)$ of $A\cap X_2$ in $\mathfrak{M}_{X_2}$.

The \emph{Tutte polynomial} of the matroid $\mathfrak{M}_X=(X,rk)$
is defined as
$$T_X(x,y):= \sum_{A\subseteq X} (x-1)^{rk(X)-rk(A)} (y-1)^{|A|-rk(A)}.$$

\subsection{Arithmetic matroids}

An \textit{arithmetic matroid} is a pair $(\mathfrak{M}_X,m)$,
where $\mathfrak{M}_X$ is a matroid on a list of vectors $X$, and
$m$ is a \textit{multiplicity function}, i.e.
$m:\mathbb{P}(X)\rightarrow \mathbb{N}\setminus \{0\}$ has the
following properties:
\begin{itemize}
    \item[(1)] if $A\subseteq X$ and $v\in X$ is dependent on $A$, then
    $m(A\cup\{v\})$ divides $m(A)$;
    \item[(2)] if $A\subseteq X$ and $v\in X$ is independent on $A$, then
    $m(A)$ divides $m(A\cup\{v\})$;
    \item[(3)] if $A\subseteq B\subseteq X$ and $B$ is a disjoint union $B=A\cup F\cup T$
    such that for all $A\subseteq C\subseteq B$  we have $rk(C)=rk(A)+|C\cap F|$, then
    $$
    m(A)\cdot m(B) = m(A\cup F)\cdot m(A\cup T).
    $$
    \item[(4)] if $A\subseteq B$ and $rk(A)=rk(B)$, then
    $$
    \mu_B(A):=\sum_{A\subseteq T\subseteq B}(-1)^{|T|-|A|}m(T)\geq 0.
    $$
    \item[(5)] if $A\subseteq B$ and $rk^*(A)=rk^*(B)$, then
    $$
    \mu_{B}^*(A):=\sum_{A\subseteq T\subseteq B}(-1)^{|T|-|A|}m(X\setminus T)\geq 0.
    $$
\end{itemize}
\begin{Ex} \label{ex:arithmatroid}
The prototype of an arithmetic matroid is the one that we are
going to associate now to a finite list $X$ of elements of a
finitely generated abelian group $G$.

Given a sublist $A\subseteq X$, we will denote by $\langle
A\rangle$ the subgroup of $G$ generated by the underlying set of
$A$.

We define the rank of a sublist $A\subseteq X$ as the maximal rank
of a free (abelian) subgroup of $\langle A\rangle$. This defines a
matroid structure on $X$.

For $A\subseteq X$, let $G_A$ be the maximal subgroup of $G$ such
that $\langle A\rangle\leq G_A$ and $|G_A: \langle
A\rangle|<\infty$, where $|G_A: \langle A\rangle|$ denotes the
index (as subgroup) of $\langle A\rangle$ in $G_A$. Then the
multiplicity $m(A)$ is defined as $m(A):=|G_A: \langle A\rangle|$.
\end{Ex}

\begin{Rem} \label{rm:GCD}
Notice that in $\mathbb{Z}^{m}$, to compute the multiplicity of a
list of elements, it is enough to see the elements as the columns
of a matrix, and to compute the greatest common divisor of its
minors of order the rank of the matrix (cf. \cite[Theorem
2.2]{StL0}).
\end{Rem}

Given an arithmetic matroid $(\mathfrak{M}_X,m)$ and a vector
$v\in X$, we define the \textit{deletion} of $(\mathfrak{M}_X,m)$
as the arithmetic matroid $(\mathfrak{M}_{X_1},m_1)$, where
$\mathfrak{M}_{X_1}$ is the deletion of $\mathfrak{M}_{X}$ and
$m_1(A):=m(A)$ for all $A\subseteq X_1=X\setminus \{v\}$.

Given an arithmetic matroid $(\mathfrak{M}_X,m)$ and a vector
$v\in X$, we define the \textit{contraction} of
$(\mathfrak{M}_X,m)$ as the arithmetic matroid
$(\mathfrak{M}_{X_2},m_2)$, where $\mathfrak{M}_{X_2}$ is the
contraction of $\mathfrak{M}_{X}$ and $m_2(A):=m(A\cup \{v\})$ for
all $A\subseteq X_2=X\setminus \{v\}$.

We say that $v\in X$ is:
\begin{itemize}
  \item \textit{free} if both
$rk_1(X\setminus\{v\})=rk(X\setminus\{v\})=rk(X)-1$ and
$rk_2(X\setminus \{v\})=rk(X)-1$;
  \item \textit{torsion} if both
$rk_1(X\setminus\{v\})=rk(X)$ and $rk_2(X\setminus \{v\})=rk(X)$;
  \item  \textit{proper} if both
$rk_1(X\setminus\{v\})=rk(X)$ and $rk_2(X\setminus
\{v\})=rk(X)-1$.
\end{itemize}

Observe that any vector of a matroid is of one and only one of the
previous three types.

\begin{Ex}
In the Example \ref{ex:arithmatroid}, the deletion consists of
removing the vector from the list, while the contraction consists
of removing it from the list, taking the quotient by the subgroup
that it generates, and identifying the remaining vectors with
their cosets.

In this case the torsion vectors are the torsion elements in the
algebraic sense, while the free vectors are the elements $v\in X$
such that $\langle X\rangle\cong \langle X\setminus \{v\}\rangle
\oplus \langle v\rangle$.
\end{Ex}

Given two arithmetic matroids $(\mathfrak{M}_{X_1},m_1)$ and
$(\mathfrak{M}_{X_2},m_2)$ we define their \textit{direct sum} as
the arithmetic matroid $(\mathfrak{M}_{X},m)$, where
$\mathfrak{M}_{X}:=\mathfrak{M}_{X_1}\oplus \mathfrak{M}_{X_2}$,
and for any sublist $A\subseteq X=X_1\sqcup X_2$, we set
$m(A):=m_1(A\cap X_1)\cdot m_2(A\cap X_2)$.
\begin{Rem} \label{rm:directsum}
If the two arithmetic matroids are represented by a list $X_1$ of
elements of a group $G_1$ and a list $X_2$ of elements of a group
$G_2$, then, with the obvious identifications, $X:=X_1\sqcup X_2$
is a list of elements of the group $G:=G_1\oplus G_2$, and the
arithmetic matroid associated to this list is exactly the direct
sum of the two.
\end{Rem}
We associate to an arithmetic matroid $\mathfrak{M}_X$ its
\textit{arithmetic Tutte polynomial}
$M_X(x,y)=M(\mathfrak{M}_X;x,y)$ defined as
\begin{equation} \label{eq:aritutte}
M_X(x,y):=\sum_{A\subseteq
X}m(A)(x-1)^{rk(X)-rk(A)}(y-1)^{|A|-rk(A)}.
\end{equation}

\subsection{Basic properties}

We summarize in the following theorem some basic properties of
this polynomial (cf. \cite[Proofs of Lemmas 6.4, 6.6, 7.6 and
Section 3.6]{DM2}).

\begin{Thm} \label{thm:properties}
Let $\mathfrak{M}_X$ be an arithmetic matroid, and let $v\in X$ be
a vector. Denote by $M_{X}(x,y)$, $M_{X_1}(x,y)$ and
$M_{X_2}(x,y)$ the arithmetic Tutte polynomial associated to
$\mathfrak{M}_X$, the deletion of $v$ and the contraction of $v$
respectively.
\begin{enumerate}
    \item If $v$ is a proper vector then
    $$
    M_{X}(x,y)=M_{X_1}(x,y)+M_{X_2}(x,y).
    $$
    \item If $v$ is a free vector then
    $$
    M_{X}(x,y)=(x-1)M_{X_1}(x,y)+M_{X_2}(x,y).
    $$
    \item If $v$ is a torsion vector then
    $$
    M_{X}(x,y)=M_{X_1}(x,y)+(y-1)M_{X_2}(x,y).
    $$
    \item If $\mathfrak{M}_X$ is the direct sum of two matroids
    $\mathfrak{M}_1$ and $\mathfrak{M}_2$, then
    $$
    M_{X}(x,y)=M(\mathfrak{M}_1;x,y)\cdot M(\mathfrak{M}_2;x,y).
    $$
\end{enumerate}
\end{Thm}

\subsection{A fundamental construction}

We associate to each labelled graph $(\mathcal{G},\ell)$ an
arithmetic matroid $\mathfrak{M}_{\mathcal{G},\ell}$ in the
following way.

First of all we enumerate the vertices $V=\{v_1,v_2,\dots,v_n\}$
and we fix an orientation $E_{\theta}$ of the edges $E$. Then to
each edge $e=(v_i,v_j)\in E_{\theta}$ we associate the element
$x_{e}\in \mathbb{Z}^n$ defined as the vector whose $i$-th
coordinate is $\ell(e)$, whose $j$-th coordinate is $-\ell(e)$,
and whose other coordinates are $0$. We denote by $X_R$ and $X_D$
the multisets of vectors in $\mathbb{Z}^n$ corresponding to
elements of $R$ and $D$ respectively.

Then we look at the group $G:=\mathbb{Z}^n/ \langle X_D\rangle$,
and we identify the elements of $X_R$ with the corresponding
cosets in $G$. This gives as an arithmetic matroid
$\mathfrak{M}=\mathfrak{M}_{X_R}=\mathfrak{M}_{\mathcal{G},\ell}$,
which is clearly independent on the orientation that we choose
(changing the orientation of an edge corresponds to multiply the
corresponding vector in $X_R$ or $X_D$ by $-1$).

We denote by $M_{X_R}(x,y)=M_{\mathcal{G},\ell}(x,y)$ the
associated arithmetic Tutte polynomial.

\begin{Ex}
Consider $(\mathcal{G},\ell)$ as in Example \ref{ex:labelgraph}
and the orientation shown in Figure 5.

\begin{figure}[h]
\includegraphics[width=60mm,clip=true,trim=10mm 180mm 30mm 10mm]{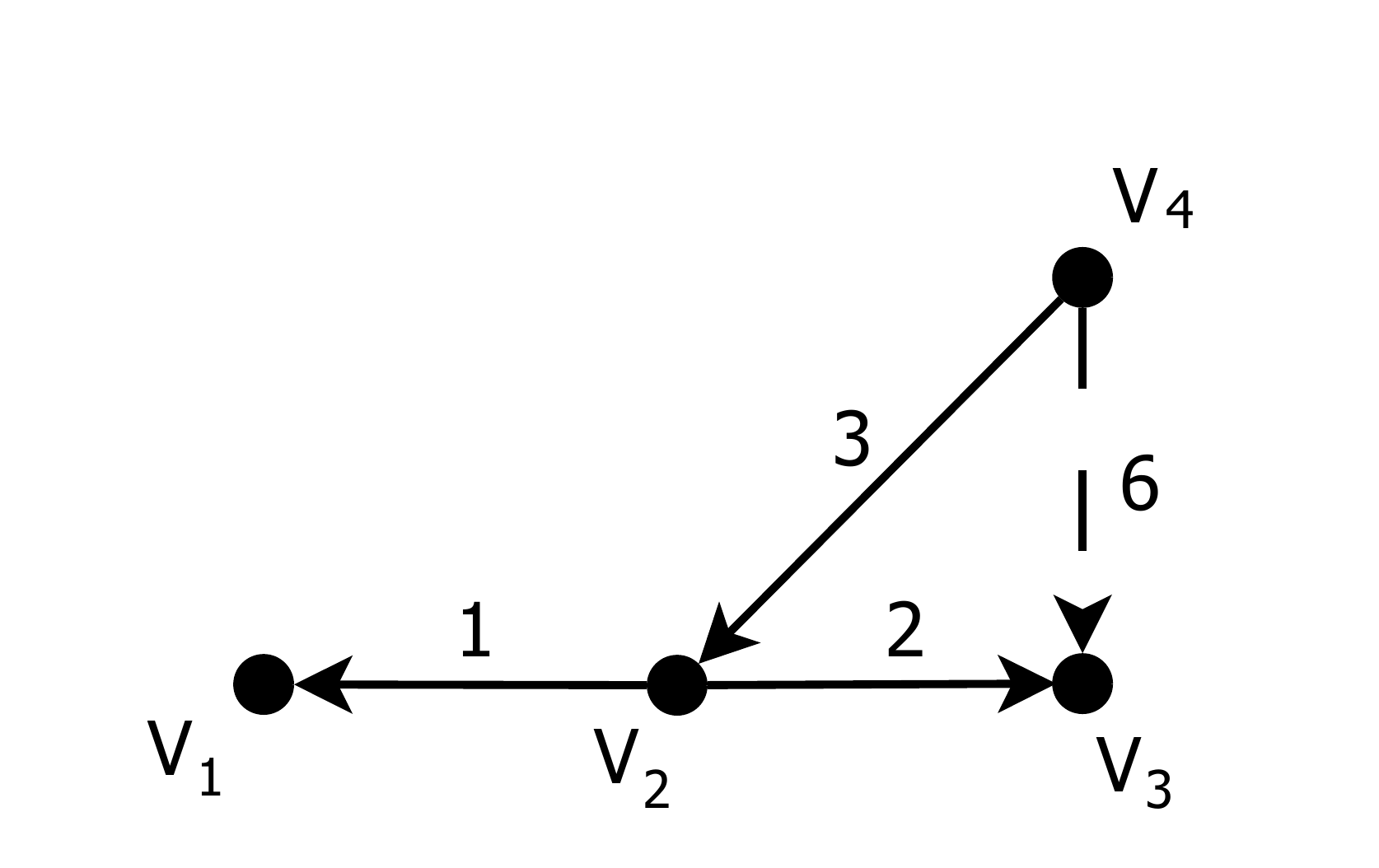}
\caption{The labelled graph $(\mathcal{G},\ell)$ with an
orientation.}
\end{figure}

We have $X_R=\{(1,-1,0,0),(0,-2,2,0),(0,3,0,-3)\}$$\subseteq
\mathbb{Z}^4$ and $X_D=\{(0,0,6,-6)\}$ $\subseteq \mathbb{Z}^4$,
so $G:=\mathbb{Z}^4/\langle (0,0,6,-6)\rangle$, and we identify
$X_R$ with $\left\{ v_1:=\overline{(1,-1,0,0)}\right.$,
$v_2:=\overline{(0,-2,2,0)}$,
$\left.v_3:=\overline{(0,3,0,-3)}\right\} \subseteq G$, where
$\overline{u}$ indicates the coset of the representative $u\in
\mathbb{Z}^4$.

We have (cf. Remark \ref{rm:GCD})
$m(\emptyset)=m(\{v_1\})=m(\{v_2,v_3\})=m(X_R)=6$,
$m(\{v_2\})=m(\{v_1,v_2\})=12$, $m(\{v_3\})=m(\{v_1,v_3\})=18$, so
$M_{\mathcal{G},\ell}(x,y)=6x^2+18x+6xy$.%, and hence
%$$
%\chi_{\mathcal{G},\ell}(q)=qM_{\mathcal{G},\ell}(1-q,0)=6q^3-30q^2+24q
%$$
%and
%$$
%{\chi}_{\mathcal{G},\ell}^*(q)=-q^2M_{\mathcal{G},\ell}(0,1-q)=0.
%$$
\end{Ex}

\section{Arithmetic colorings}

In this section we discuss the notion of arithmetic coloring.

\subsection{Definitions}

Given a labelled graph $(\mathcal{G},\ell)$, let $q$ be a positive
integer. An \textit{arithmetic (proper) $q$-coloring} of
$(\mathcal{G},\ell)$ is a map $c:V\rightarrow
\mathbb{Z}/q\mathbb{Z}$ that satisfies the following conditions:
\begin{enumerate}
    \item if $u,v\in V$ and $e:=\{u,v\}\in R$, then $\ell(e)\cdot c(u)\neq \ell(e)\cdot
    c(v)$;
    \item if $u,v\in V$ and $e:=\{u,v\}\in D$, then $\ell(e)\cdot c(u)= \ell(e)\cdot
    c(v)$.
\end{enumerate}

For our results we will need to restrict ourself to consider only
positive integers $q$ such that $\ell(e)$ divides $q$ for all
$e\in E$ (cf. Remark \ref{rm:conditioncolor}). We will call such
an integer \textit{admissible}.
\begin{Rem}
For a trivial labelling $\ell\equiv 1$ and $D=\emptyset$, clearly
we just recover the usual notion of \textit{(proper) $q$-coloring}
(any $q$ will be admissible now) of the underlying graph.

If $\ell\equiv 1$ and $D\neq \emptyset$, then we can still
interpret it as a $q$-coloring, but this time of the graph
$\overline{\mathcal{G}}$ (that is obtained from $\mathcal{G}$ by
performing the classical contraction of all the edges in $D$).

More generally, given $\ell(e)=1$ for some $e\in D$, if we do a
classical contraction of $e$, then we get a graph with the same
number of arithmetic $q$-coloring.
\end{Rem}
\begin{Ex} \label{ex:color}
Consider $(\mathcal{G},\ell)$ with $\mathcal{G}:=(V,E)$,
$V=\{v_1,v_2,v_3\}$, $R:=\{e_1:=\{v_1,v_2\},e_2:=\{v_1,v_2\}\}$,
$D:=\{e_3:=\{v_2,v_3\}\}$ so $E=R\cup D=\{e_1,e_2,e_3\}$,
$\ell(e_1)=2$, $\ell(e_2)=3$ and $\ell(e_3)=2$ (see Figure 6).

\begin{figure}[h]
\includegraphics[width=60mm,clip=true,trim=10mm 230mm 30mm 10mm]{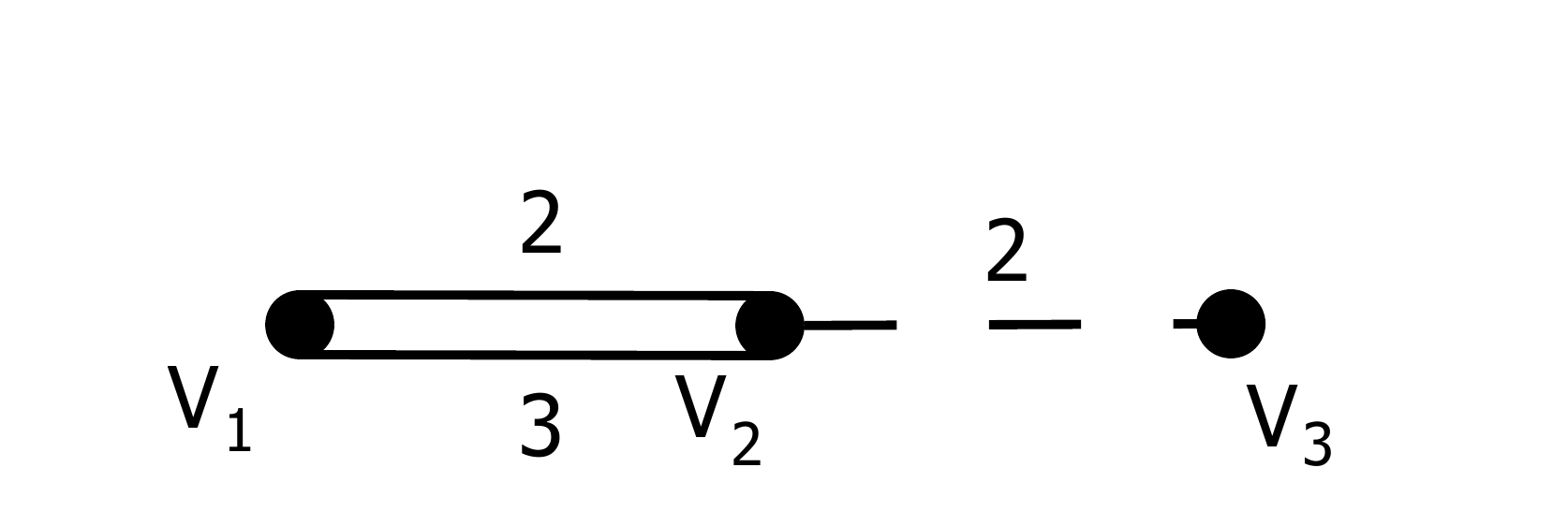}
\caption{The labelled graph $(\mathcal{G},\ell)$.}
\end{figure}

Then any multiple of $6$ is admissible. For example for $q=6$ we
denote the $6$-colorings as vectors of
$(\mathbb{Z}/6\mathbb{Z})^3$, where for every $i=1,2,3$, the
$i$-th coordinate corresponds to the color of the vertex $v_i$.

In this case there are $24$ possible $6$-colorings of
$(\mathcal{G},\ell)$:
$(\overline{a},\overline{1}+\overline{a},\overline{1}+\overline{a})$,
$(\overline{a},\overline{1}+\overline{a},\overline{4}+\overline{a})$,
$(\overline{a},\overline{5}+\overline{a},\overline{5}+\overline{a})$,
$(\overline{a},\overline{5}+\overline{a},\overline{2}+\overline{a})$
for all $\overline{a}\in \mathbb{Z}/6\mathbb{Z}$.
\end{Ex}

We define the \textit{arithmetic chromatic polynomial} of a
labelled graph $(\mathcal{G},\ell)$ to be the function
${\chi}_{\mathcal{G},\ell}(\cdot): \mathbb{N}\setminus
\{0\}\rightarrow \mathbb{N}\cup \{0\}$, which assigns to each
positive integer $q$ the number of arithmetic $q$-colorings of
$(\mathcal{G},\ell)$. We will show in Theorem \ref{thm:maincolor}
that this is in fact a polynomial function.

\begin{Rem} \label{rm:trivialcolor}
For a trivial labelling $\ell\equiv 1$ and $D=\emptyset$, we just
recover the usual notion of the \textit{chromatic polynomial}
$\chi_{\mathcal{G}}(q)$ of the underlying graph $\mathcal{G}$.

If $\ell\equiv 1$ and $D\neq \emptyset$, then we can still
interpret it as a chromatic polynomial of the graph
$\overline{\mathcal{G}}$.
\end{Rem}

\begin{Ex}
Consider $(\mathcal{G},\ell)$ as in Example \ref{ex:color}. We
have $q$ choices for the color $c(v_1)$ of $v_1$, $q-4$ choices
for $c(v_2)$ (all except $c(v_1)$, $c(v_1)+\overline{q/2}$,
$c(v_1)+\overline{q/3}$, $c(v_1)+\overline{2q/3}$) and $2$ choices
for $c(v_3)$ ($c(v_2)$ and $c(v_2)+\overline{q/2}$). Hence
$\chi_{\mathcal{G},\ell}(q)=2q(q-4)=2q^2-8q$, which agrees with
what we found for $q=6$ ($\chi_{\mathcal{G},\ell}(6)=24$).
\end{Ex}

\subsection{Main result}

We state the main result of this section.

\begin{Thm} \label{thm:maincolor}
Let $(\mathcal{G},\ell)$ be a labelled graph and let $q$ be an
admissible (positive) integer, i.e. $\ell(e)$ divides $q$ for all
$e\in E$. Let $\frak{M}_{\mathcal{G},\ell}$ be the associated
arithmetic matroid, and let $M_{\mathcal{G},\ell}(x,y)$ be the
associated arithmetic Tutte polynomial. If $k$ is the number of
connected components of the graph $\mathcal{G}$, then
$$
{\chi}_{\mathcal{G},\ell}(q)=\widetilde{{\chi}}_{\mathcal{G},\ell}(q):=(-1)^{|\overline{V}|-k}q^k
M_{\mathcal{G},\ell}(1-q,0).
$$
\end{Thm}
\begin{Ex}
Consider $(\mathcal{G},\ell)$ as in Example \ref{ex:color}.

Let us construct the associated arithmetic matroid: fix the
orientation $R_{\theta}:=\{e_1:=(v_1,v_2),e_2:=(v_2,v_1)\}$,
$D_{\theta}:=\{e_3:=(v_2,v_3)\}$ so that $E_{\theta}=R_{\theta}
\cup D_{\theta}$ (see Figure 7).

\begin{figure}[h]
\includegraphics[width=60mm,clip=true,trim=10mm 230mm 30mm 10mm]{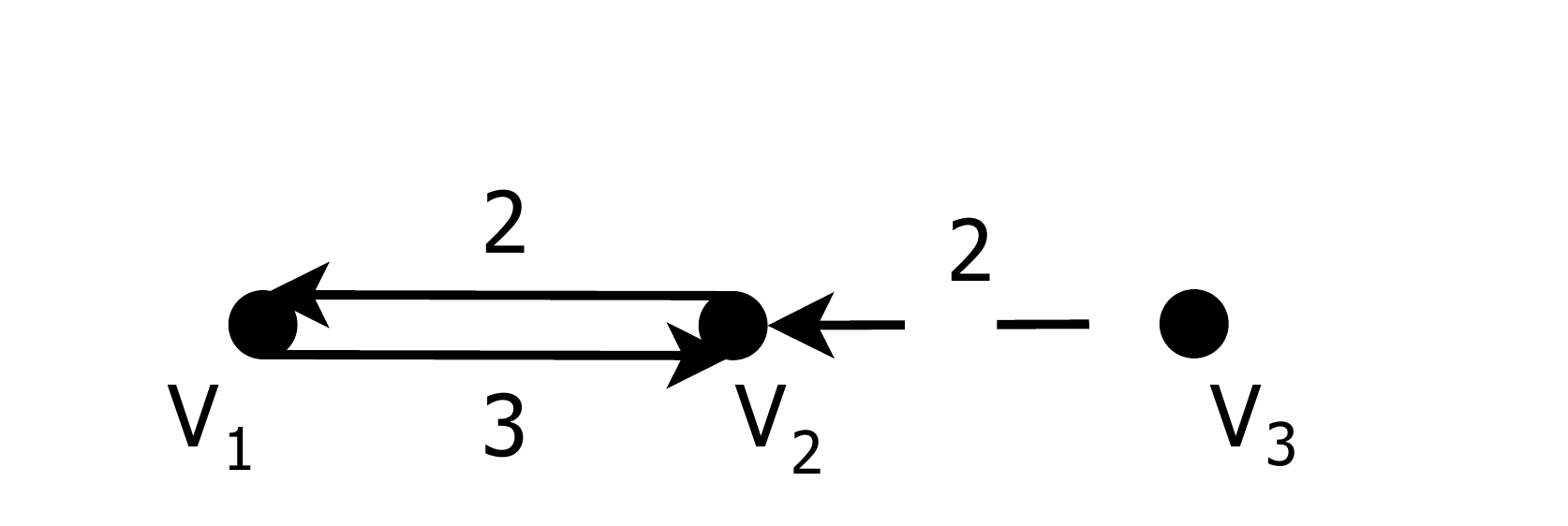}
\caption{The labelled graph $(\mathcal{G},\ell)$ with orientation
$E_{\theta}$.}
\end{figure}

Hence $X_D=\left\{(0,2,-2)\right\}$ $\subseteq \mathbb{Z}^3$, and
$X_R\equiv \left\{
\overline{(2,-2,0)},\overline{(-3,3,0)}\right\}\subseteq
G:=\mathbb{Z}^{3}/\langle X_D\rangle=\mathbb{Z}^{3}/\langle
(0,2,-2)\rangle$. An easy computation shows that
$M_{\mathcal{G},\ell}(x,y)=2x+6+2y$, and therefore
$$
\widetilde{\chi}_{\mathcal{G},\ell}(q)=-qM_{\mathcal{G},\ell}(1-q,0)
=2q^2-8q=\chi_{\mathcal{G},\ell}(q),
$$
as predicted.
\end{Ex}
\begin{Rem} \label{rm:conditioncolor}
The admissibility condition on $q$ (i.e. $\ell(e)$ divides it for
all $e\in E$) is necessary: for example, consider
$(\mathcal{G},\ell)$, where $\mathcal{G}:=(V,E)$ with
$V:=\{v_1,v_2,v_3\}$, $R:=\emptyset$,
$D:=\{\{v_1,v_2\},\{v_2,v_3\}\}$ so that $E=R\cup D$, and
$\ell(\{v_1,v_2\})=2$, $\ell(\{v_2,v_3\})=6$.

For $q=2$, the conditions on the colors are trivially satisfied,
hence we have $2^3=8$ arithmetic $2$-colorings.

Let us construct the associated arithmetic matroid: fix the
orientation $E_{\theta}=\{(v_1,v_2)$, $(v_3,v_2)\}$ (see Figure
8).

\begin{figure}[h]
\includegraphics[width=60mm,clip=true,trim=10mm 230mm 30mm 10mm]{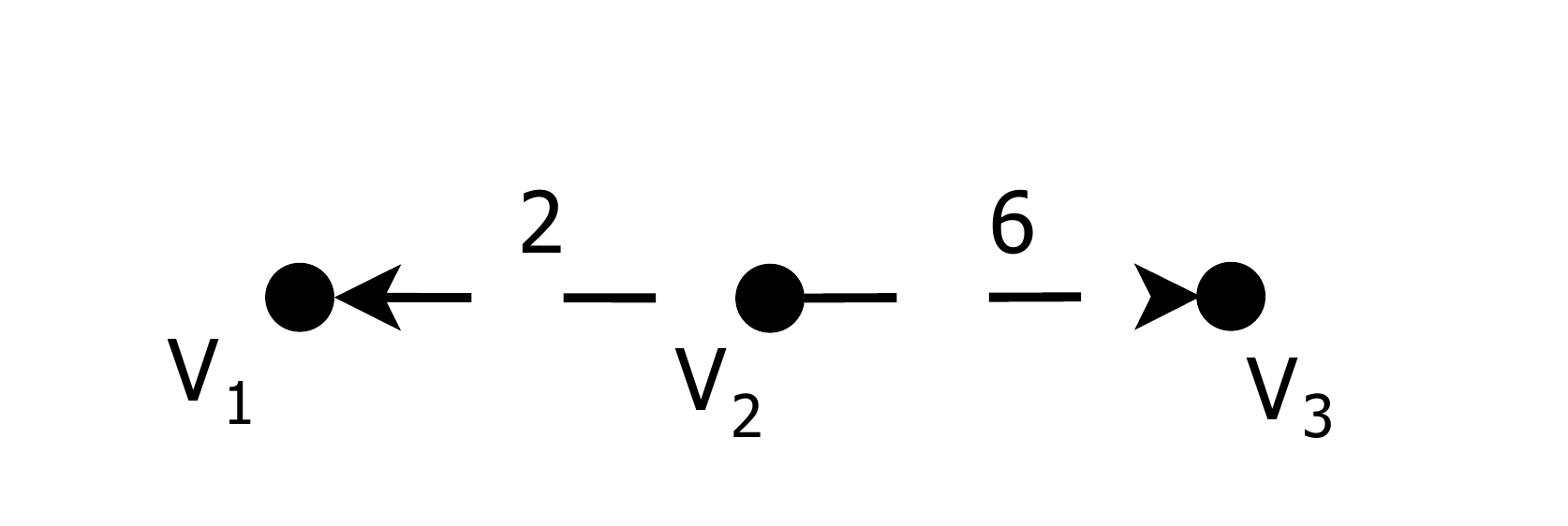}
\caption{The labelled graph $(\mathcal{G},\ell)$ with orientation
$E_{\theta}$.}
\end{figure}

Hence $X_D=\left\{(2,-2,0)\right.$, $\left.(0,-6,6)\right\}$
$\subseteq \mathbb{Z}^3$ and $X_R=\emptyset\subseteq
G:=\mathbb{Z}^{3}/\langle X_D\rangle=\mathbb{Z}^{3}/\langle
(2,-2,0),(0,-6,6)\rangle$. An easy computation gives
$M_{\mathcal{G},\ell}(x,y)=12$, and therefore
$$
\widetilde{\chi}_{\mathcal{G},\ell}(q)=qM_{\mathcal{G},\ell}(1-q,0)=12q.
$$
But then
$$
\widetilde{\chi}_{\mathcal{G},\ell}(2)=12\cdot 2=24\neq
8=\chi_{\mathcal{G},\ell}(2).
$$
\end{Rem}

If $\ell\equiv 1$, then we can identify the graph $\mathcal{G}$
with $\overline{\mathcal{G}}$ (cf. Remark \ref{rm:trivialcolor}).
In this case all the multiplicities in
$\mathfrak{M}_{\mathcal{G},\ell}$ are equal to $1$, therefore we
recover the following classical result of Tutte \cite{Tu} (cf.
also \cite[Proposition 6.3.1]{Wh}) as a special case of Theorem
\ref{thm:maincolor}.
\begin{Cor} \label{cor:color}
We have
$$
\chi_{\mathcal{G}}(q)=(-1)^{|V|-k}q^kT_{\mathcal{G}}(1-q,0),
$$
where $\chi_{\mathcal{G}}(q)$ is the chromatic polynomial of the
graph $\mathcal{G}=(V,E)$, $k$ is the number of connected
components of $\mathcal{G}$ and $T_{\mathcal{G}}(x,y)$ is the
associated Tutte polynomial.
\end{Cor}

\section{Proof of Theorem \ref{thm:maincolor}}

To prove Theorem \ref{thm:maincolor} we need the following lemma,
which is immediate from the definitions.
\begin{Lem}
Let $(\mathcal{G},\ell)$ be a labelled graph and let $q$ be an
admissible integer. For a regular edge $e\in R$ we have
$$
{\chi}_{\mathcal{G},\ell}(q)= {\chi}_{\mathcal{G}- e,\ell}(q)-
{\chi}_{\mathcal{G}/e,\ell}(q).
$$
\end{Lem}

We want to prove that our polynomial
$\widetilde{{\chi}}_{\mathcal{G},\ell}(q)$ satisfies the same
recursion.
\begin{Lem}
Let $(\mathcal{G},\ell)$ be a labelled graph and let $q$ be an
admissible integer. For a regular edge $e\in R$ we have
$$
\widetilde{{\chi}}_{\mathcal{G},\ell}(q)=
\widetilde{{\chi}}_{\mathcal{G}- e,\ell}(q)-
\widetilde{{\chi}}_{\mathcal{G}/e,\ell}(q).
$$
\end{Lem}
\begin{proof}
We distinguish three cases.

\underline{Case 1}: $e$ is a proper edge, i.e. the corresponding
edge in $\overline{\mathcal{G}}$ is not a loop and it is contained
in a circuit. Then, applying Theorem \ref{thm:properties} (1), we
have
\begin{eqnarray*}
\widetilde{{\chi}}_{\mathcal{G},\ell}(q) & = &
(-1)^{|\overline{V}|-k}q^kM_{\mathcal{G}}(1-q,0)\\
 & = &(-1)^{|\overline{V}|-k}q^k(M_{\mathcal{G}- e}(1-q,0)+
 M_{\mathcal{G}/e}(1-q,0))\\
 & = & (-1)^{|\overline{V}_1|-k}q^{k}M_{\mathcal{G}- e}(1-q,0)-
 (-1)^{|\overline{V}_2|-k}q^{k}M_{\mathcal{G}/e}(1-q,0)\\
 & = & \widetilde{{\chi}}_{\mathcal{G}- e,\ell}(q)-
\widetilde{{\chi}}_{\mathcal{G}/e,\ell}(q),
\end{eqnarray*}
since $|\overline{V}_1|=|\overline{V}|$ and
$|\overline{V}_2|=|\overline{V}|-1$.

\underline{Case 2}: $e$ is a free edge, i.e. the corresponding
edge in $\overline{\mathcal{G}}$ is not contained in a circuit and
is not a loop. Then, applying Theorem \ref{thm:properties} (2), we
have
\begin{eqnarray*}
\widetilde{{\chi}}_{\mathcal{G},\ell}(q) & = &
(-1)^{|\overline{V}|-k}q^kM_{\mathcal{G}}(1-q,0)\\
 & = &(-1)^{|\overline{V}|-k}q^k(-qM_{\mathcal{G}- e}(1-q,0)+
 M_{\mathcal{G}/e}(1-q,0))\\
 & = & (-1)^{|\overline{V}_1|-(k+1)}q^{k+1}M_{\mathcal{G}- e}(1-q,0)-
 (-1)^{|\overline{V}_2|-k}q^kM_{\mathcal{G}/e}(1-q,0)\\
 & = & \widetilde{{\chi}}_{\mathcal{G}- e,\ell}(q)-
\widetilde{{\chi}}_{\mathcal{G}/e,\ell}(q),
\end{eqnarray*}
since $\mathcal{G}- e$ has now one extra connected component,
$|\overline{V}_1|=|\overline{V}|$, and
$|\overline{V}_2|=|\overline{V}|-1$.

\underline{Case 3}: $e$ is a torsion edge, i.e. the corresponding
edge in $\overline{\mathcal{G}}$ is a loop. Then, applying Theorem
\ref{thm:properties} (3), we have
\begin{eqnarray*}
\widetilde{{\chi}}_{\mathcal{G},\ell}(q) & = &
(-1)^{|\overline{V}|-k}q^kM_{\mathcal{G}}(1-q,0)\\
 & = &(-1)^{|\overline{V}|-k}q^k(M_{\mathcal{G}- e}(1-q,0)-
 M_{\mathcal{G}/e}(1-q,0))\\
 & = & (-1)^{|\overline{V}_1|-k}q^{k}M_{\mathcal{G}- e}(1-q,0)-
 (-1)^{|\overline{V}_2|-k}q^kM_{\mathcal{G}/e}(1-q,0)\\
 & = & \widetilde{{\chi}}_{\mathcal{G}- e,\ell}(q)-
\widetilde{{\chi}}_{\mathcal{G}/e,\ell}(q),
\end{eqnarray*}
since $|\overline{V}_1|=|\overline{V}|$ and
$|\overline{V}_2|=|\overline{V}|$.
\end{proof}

In this way we reduce the proof of Theorem \ref{thm:maincolor} to
the case where there are no regular edges. For this case, first of
all we reduce ourself to the case of a connected graph: suppose
that our graph $\mathcal{G}$ has $k$ connected components
$\mathcal{G}^{(1)},\mathcal{G}^{(2)},\dots,\mathcal{G}^{(k)}$ with
the corresponding labellings
$\ell^{(1)},\ell^{(2)},\dots,\ell^{(k)}$. In this case the matroid
$\mathfrak{M}_{\mathcal{G},\ell}$ is the direct sum of the
matroids $\mathfrak{M}_{\mathcal{G}^{(1)},\ell^{(1)}},
\mathfrak{M}_{\mathcal{G}^{(2)},\ell^{(2)}},\dots,
\mathfrak{M}_{\mathcal{G}^{(k)},\ell^{(k)}}$ (cf. Remark
\ref{rm:directsum}). Since $\overline{\mathcal{G}^{(i)}}$ consists
of a single vertex with no edges for $i=1,2,\dots,k$, we have
$|\overline{V}|=k$. Therefore, assuming the result for a connected
graph, we have
\begin{eqnarray*}
(-1)^{|\overline{V}|-k}q^kM_{\mathcal{G},\ell}(1-q,0) & = &
q^kM_{\mathcal{G},\ell}(1-q,0)\\
\text{(by Theorem \ref{thm:properties} (4))} & = &
q^k\cdot\prod_{i=1}^k
M_{\mathcal{G}^{(i)},\ell^{(i)}}(1-q,0)\\
 & = & \prod_{i=1}^k q\cdot
M_{\mathcal{G}^{(i)},\ell^{(i)}}(1-q,0)\\
\text{(by assumption on connected graphs)} & = & \prod_{i=1}^k
\chi_{\mathcal{G}^{(i)},\ell^{(i)}}(q)=
\chi_{\mathcal{G},\ell}(q),
\end{eqnarray*}
where the last equality is clear from the definition of arithmetic
chromatic polynomial.

So we are left to prove the following lemma.
\begin{Lem}
Let $(\mathcal{G},\ell)$ be a labelled connected graph with no
regular edges and let $q$ be an admissible integer. We have
$$
\chi_{\mathcal{G},\ell}(q)=\widetilde{{\chi}}_{\mathcal{G},\ell}(q).
$$
\end{Lem}
\begin{proof}
First of all notice that
$$
\widetilde{{\chi}}_{\mathcal{G},\ell}(q)=qM_{\mathcal{G},\ell}(1-q,0)
=q\cdot m(\emptyset),
$$
where $m(\emptyset)$ is the cardinality of the torsion subgroup of
the group $\mathbb{Z}^{|V|}/\langle X_D\rangle$, i.e. the GCD of
the maximal rank nonzero minors of the matrix $[X_D]$ whose
columns are the elements of $X_D$ (cf. Remark \ref{rm:GCD}).

To clarify the general idea, let us start with the special case of
$\mathcal{G}$ being a tree.

The number of arithmetic coloring in this case is $q$ times the
product of the labels involved. To see this we can proceed by
induction: when there is only one vertex is clear; when we add an
edge $e$ with a label $\ell(e)$ we simply have $\ell(e)$ many
choices for the new vertex. Here we are using that $\ell(e)$
divides $q$ for every edge $e\in E$ ($=D$ in this case), since in
general we would have $GCD(\ell(e),q)$ many choices for the new
edge $e$.

In this case ($\mathcal{G}$ is a tree) $m(\emptyset)$ is just the
product of the labels of the edges: in fact we have an independent
set, hence we are just computing the determinant of the matrix
with the vectors in $X_D$. The fact that we get the product of the
labels is now clear from the form of the matrix $[X_D]$.

Let us consider now the general case, where $\mathcal{G}$ is not
necessarily a tree.

To compute the number of arithmetic $q$-colorings, let
$m:=|X_D|=|D|$. If $n=|V|$, then $m\geq n-1$ (we are assuming that
$\mathcal{G}$ is connected). We can look at $[X_D]$ as a linear
operator (acting on the right)
$$
[X_D]:(\mathbb{Z}/q\mathbb{Z})^n\rightarrow
(\mathbb{Z}/q\mathbb{Z})^m.
$$
Then the elements of the kernel correspond naturally to the
arithmetic $q$-colorings, hence we need to compute the cardinality
of $Ker [X_D]$. Looking at the exact sequence (i.e. the image of
each homomorphism is the kernel of the successive one)
$$
0 \rightarrow Ker [X_D]\mathop{\hookrightarrow}^{\iota}
(\mathbb{Z}/q\mathbb{Z})^n \mathop{\longrightarrow}^{[X_D]}
(\mathbb{Z}/q\mathbb{Z})^m \mathop{\longrightarrow}^{\pi}
\frac{(\mathbb{Z}/q\mathbb{Z})^m}{Im [X_D]}\rightarrow 0,
$$
where $\iota$ and $\pi$ are the inclusion and the natural map to
the quotient respectively, we realize that
$(\mathbb{Z}/q\mathbb{Z})^n/Ker [X_D]\cong Im [X_D]$ implies
$$
\frac{|(\mathbb{Z}/q\mathbb{Z})^m|}{|Im
[X_D]|}=\frac{|(\mathbb{Z}/q\mathbb{Z})^m|}{|(\mathbb{Z}/q\mathbb{Z})^n|}\cdot
|Ker [X_D]|= q^{m-n} \cdot |Ker [X_D]|.
$$
Now by the isomorphism theorem
$$
|(\mathbb{Z}/q\mathbb{Z})^m:Im [X_D]|=|\mathbb{Z}^m:Im [X_D\cup
Q]|,
$$
where $Q=\{q_1,q_2,\dots,q_m\}$, $q_i$ is the vector of
$\mathbb{Z}^m$ with $1$ in the $i$-th position and $0$ elsewhere,
and we see the matrix $[X_D\cup Q]$ (which consists of a $n\times
m$ block on the top whose columns are the elements of $X_D$, and a
$m\times m$ block on the bottom whose columns are the elements of
$Q$) as a linear operator (acting on the right)
$$
[X_D\cup Q]:\mathbb{Z}^{m+n}\rightarrow \mathbb{Z}^{m}.
$$
But when we compute the right-hand side of the last equation  we
get $m(\emptyset)\cdot q^{m-n+1}$. In fact by Remark \ref{rm:GCD}
this the $GCD$ of the minors of maximal rank of $[X_D\cup Q]$.
But, since $\ell(e)$ divides $q$ for every edge $e\in E$ ($=D$ in
this case), it is enough to compute the minors of maximal rank
which involve the edges of a spanning tree plus some extra columns
from $Q$ (the other ones are multiples of these). Therefore we get
$q^{m-n+1}$ times the $GCD$ of minors of rank $n-1$ in $[X_D]$,
which are the ones corresponding to spanning trees. But this last
number is exactly $m(\emptyset)$.

From all this we get
$$
\chi_{\mathcal{G},\ell}(q)=|Ker([X_D])|=q\cdot m(\emptyset)=q\cdot
M_{\mathcal{G},\ell}(1-q,0)=\widetilde{\chi}_{\mathcal{G},\ell}(q),
$$
as we wanted.
\end{proof}

\section{Arithmetic flows}

In this section we discuss the notion of arithmetic flow.

\subsection{Definitions}

Given a labelled graph $(\mathcal{G},\ell)$, fix an orientation
$\theta$ and call $\mathcal{G}_{\theta}$ the associated directed
graph. Let $q$ be a positive integer, and set
$H:=\mathbb{Z}/q\mathbb{Z}$.

\bigskip

\textit{We fix the above notation throughout the all section.}

\bigskip

We associate to each oriented edge $e:=(u,v)\in E_{\theta}$ a
weight $w(e)\in H$. For a vertex $v\in V$, we set
$$
u(v):= \mathop{\sum_{e^+= v}}_{e\in E_{\theta}}\ell(e)\cdot
w(e)-\mathop{\sum_{e^-= v}}_{e\in E_{\theta}}\ell(e)\cdot w(e)\in
H.
$$
The function $w:E\rightarrow H$ is called a \textit{arithmetic
(nowhere zero) $q$-flow} if the following conditions hold:
\begin{enumerate}
    \item for every vertex $v\in V$,
    $$
    u(v)=0\in H;
    $$
    \item for every regular edge $e\in R_{\theta}$,
    $$
    w(e)\neq 0\in H.
    $$
\end{enumerate}

Again, for our results we will need to restrict ourself to
\textit{admissible} $q$'s, i.e. $\ell(e)$ divides $q$ for every
$e\in E$ (cf. Remark \ref{rm:conditionflow}).
\begin{Rem}
For a trivial labelling $\ell\equiv 1$ and $D=\emptyset$, clearly
we just recover the usual notion of \textit{(nowhere zero)
$q$-flow} (any $q$ will be admissible now) of the underlying
graph.
\end{Rem}

\begin{Ex} \label{ex:flow}
Consider $(\mathcal{G},\ell)$ as in Example \ref{ex:color}, and
fix the orientation
$R_{\theta}:=\{e_1:=(v_1,v_2),e_2:=(v_2,v_1)\}$,
$D_{\theta}:=\{e_3:=(v_2,v_3)\}$ so that $E_{\theta}=R_{\theta}
\cup D_{\theta}$ (see Figure 9).

\begin{figure}[h]
\includegraphics[width=60mm,clip=true,trim=10mm 230mm 30mm 10mm]{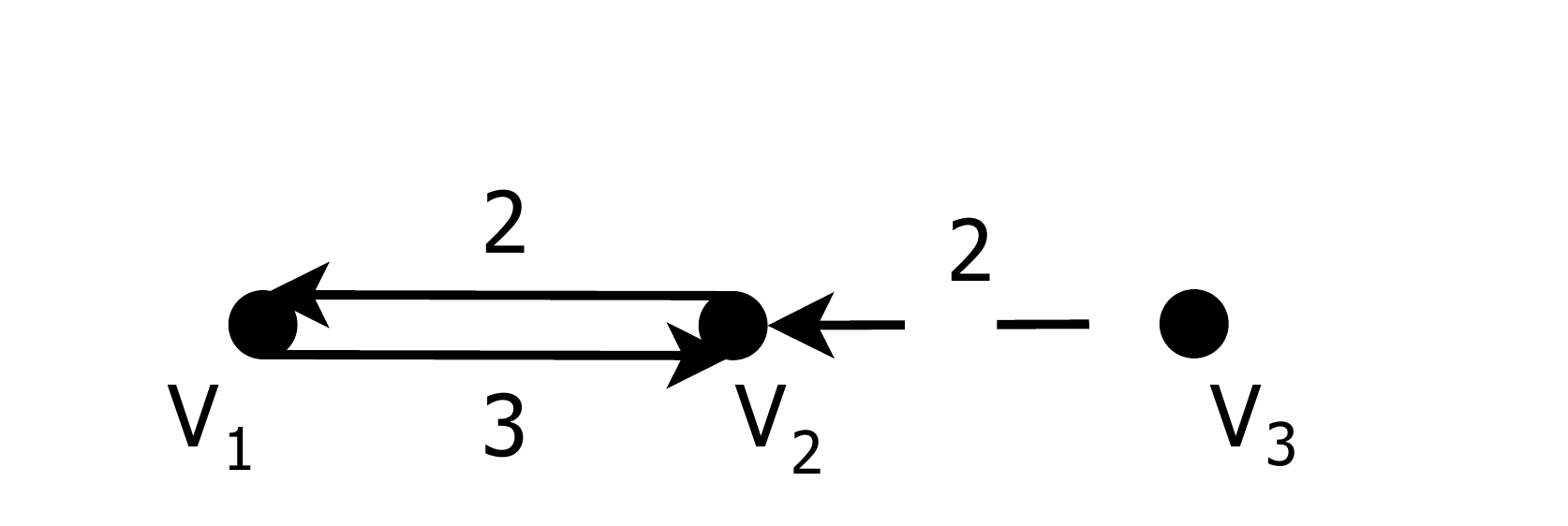}
\caption{The labelled graph $(\mathcal{G},\ell)$ with orientation
$E_{\theta}$.}
\end{figure}

Then any multiple of $6$ is admissible. For example for $q=6$ we
denote the $6$-flows as vectors of $(\mathbb{Z}/6\mathbb{Z})^3$,
where for every $i=1,2,3$, the $i$-th coordinate corresponds to
the weight of the edge $e_i$.

In this case there are $4$ possible $6$-flows of
$(\mathcal{G},\ell)$: $(\overline{3},\overline{2},\overline{0})$,
$(\overline{3},\overline{2},\overline{3})$,
$(\overline{3},\overline{4},\overline{0})$,
$(\overline{3},\overline{4},\overline{3})$.
\end{Ex}
We define the \textit{arithmetic flow polynomial} of a labelled
graph $(\mathcal{G},\ell)$ to be the function
${\chi}_{\mathcal{G},\ell}^*(\cdot): \mathbb{N}\setminus
\{0\}\rightarrow \mathbb{N}\cup \{0\}$, which assign to each
positive integer $q$ the number of arithmetic $q$-flows of
$(\mathcal{G},\ell)$. We will show in Theorem \ref{thm:mainflow}
that this is in fact a polynomial function.
\begin{Rem} \label{rm:trivialflow}
For a trivial labelling $\ell\equiv 1$ and $D=\emptyset$, clearly
we just recover the usual notion of \textit{flow polynomial}
${\chi}_{\mathcal{G}}^*(q)$ of the underlying graph $\mathcal{G}$.
\end{Rem}
\begin{Ex}
Consider $(\mathcal{G},\ell)$ as in Example \ref{ex:flow}. For the
vertex $v_3$ we have the equation $2w(e_3)=0\in H$, for which we
have only the two solutions $\overline{0}$ and $\overline{q/2}$.
In both cases they don't contribute the the conditions on the
other two vertices. Therefore, for both vertices $v_1$ and $v_2$
we get the equation $2w(e_1)-3w(e_2)=0\in H$, which has $q$
solutions: for every value of $w(e_2)$ of the form
$w(e_2)=2\overline{a}\in H$ there are exactly two values of
$w(e_1)$ which give a solution, while for the other values of
$w(e_2)$ there are no solutions.

But the inequalities $w(e_1)\neq 0\in H$ and $w(e_2)\neq 0\in H$
exclude the four possibilities $(\overline{0},\overline{0})$,
$(\overline{0},\overline{q/3})$, $(\overline{0},\overline{2q/3})$,
$(\overline{q/2},\overline{0})$ for $(w(e_1),w(e_2))$. Hence we
have $q-4$ solutions.

In conclusion, $\chi_{\mathcal{G},\ell}^*(q)=2(q-4)=2q-8$, which
agrees with what we found for $q=6$
($\chi_{\mathcal{G},\ell}^*(6)=4$).
\end{Ex}
\begin{Lem}
${\chi}_{\mathcal{G,\ell}}^*(q)$ does not depend on the
orientation that we choose.
\end{Lem}
\begin{proof}
If we change the orientation of one edge then we can always change
the sign of the corresponding weight (which is an element of $H$).
\end{proof}

\subsection{Main result}

We state the main result of this section.

\begin{Thm} \label{thm:mainflow}
Given a labelled graph $(\mathcal{G},\ell)$ with $k$ connected
components and an admissible integer $q$,
$$
{\chi}_{\mathcal{G},\ell}^*(q)=\widetilde{\chi}_{\mathcal{G},\ell}^*(q)
:=(-1)^{|R|-|\overline{V}|+k}q^{|D|-|V|+|\overline{V}|}M_{\mathcal{G},\ell}(0,1-q).
$$
\end{Thm}
\begin{Ex}
Consider $(\mathcal{G},\ell)$ with the orientation as in Example
\ref{ex:flow}.

Hence $X_D=\left\{(0,2,-2)\right\}$ $\subseteq \mathbb{Z}^3$ and
$X_R\equiv \left\{
\overline{(2,-2,0)},\overline{(-3,3,0)}\right\}\subseteq
G:=\mathbb{Z}^{3}/\langle X_D\rangle=\mathbb{Z}^{3}/\langle
(0,2,-2)\rangle$. An easy computation shows that
$M_{\mathcal{G},\ell}(x,y)=2x+6+2y$, and therefore
$$
\widetilde{\chi}_{\mathcal{G},\ell}^*(q)=-M_{\mathcal{G},\ell}(0,1-q)
=2q-8=\chi_{\mathcal{G},\ell}^*(q),
$$
as predicted.
\end{Ex}
\begin{Rem} \label{rm:conditionflow}
The admissibility condition on $q$ (i.e. $\ell(e)$ divides it for
all $e\in E$) is necessary: for example, consider
$(\mathcal{G},\ell)$, where $\mathcal{G}:=(V,E)$ with
$V:=\{v_1,v_2,v_3\}$, $R:=\emptyset$,
$D:=\{\{v_1,v_2\},\{v_2,v_3\}\}$ so that $E=R\cup D$, and
$\ell(\{v_1,v_2\})=2$, $\ell(\{v_2,v_3\})=6$.

For $q=2$, the conditions on the flows are trivially satisfied,
hence we have $2^2=4$ $2$-flows.

Recall from Remark \ref{rm:conditioncolor} the associated
arithmetic matroid for the fixed orientation
$E_{\theta}=\{(v_1,v_2)$, $(v_3,v_2)\}$ (see Figure 10).

\begin{figure}[h]
\includegraphics[width=60mm,clip=true,trim=10mm 230mm 30mm 10mm]{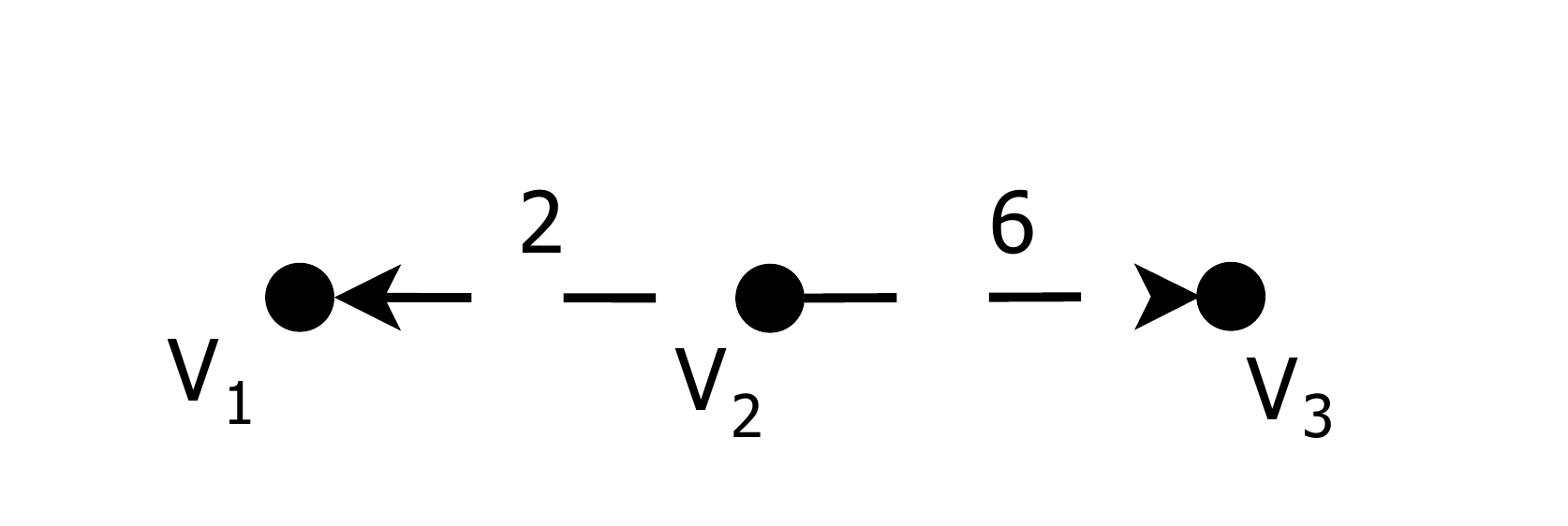}
\caption{The labelled graph $(\mathcal{G},\ell)$ with orientation
$E_{\theta}$.}
\end{figure}

We computed $M_{\mathcal{G},\ell}(x,y)=12$, therefore
$$
\widetilde{\chi}_{\mathcal{G},\ell}^*(q)=M_{\mathcal{G},\ell}(0,1-q)=12.
$$
But then
$$
\widetilde{\chi}_{\mathcal{G},\ell}^*(2)=12\neq
4=\chi_{\mathcal{G},\ell}^*(2).
$$
\end{Rem}
If $|D|=\emptyset$ and $\ell\equiv 1$, then we can identify the
graph $\mathcal{G}$ with $\overline{\mathcal{G}}$ (cf. Remark
\ref{rm:trivialflow}). In this case all the multiplicities in
$\mathfrak{M}_{\mathcal{G},\ell}$ are equal to $1$, therefore we
recover the following classical result of Tutte \cite{Tu} (cf.
also \cite[Proposition 6.3.4]{Wh}) as a special case of Theorem
\ref{thm:mainflow}.
\begin{Cor} \label{cor:flow}
We have
$$
{\chi}_{\mathcal{G}}^*(q)=(-1)^{|E|-|V|+k}T_{\mathcal{G}}(0,1-q),
$$
where ${\chi}_{\mathcal{G}}^*(q)$ is the flow polynomial of the
graph $\mathcal{G}=(V,E)$, $k$ is the number of connected
components of $\mathcal{G}$ and $T_{\mathcal{G}}(x,y)$ is the
associated Tutte polynomial.
\end{Cor}

\section{Proof of Theorem \ref{thm:mainflow}}

To prove Theorem \ref{thm:mainflow} we need the following lemma.
\begin{Lem}
Let $(\mathcal{G},\ell)$ be a labelled graph and let $q$ be an
admissible integer. For a regular edge $e\in R$ we have
$$
{\chi}_{\mathcal{G,\ell}}^*(q)={\chi}_{\mathcal{G}/e,\ell}^*(q)-
{\chi}_{\mathcal{G}- e,\ell}^*(q).
$$
\end{Lem}
\begin{proof}
It is enough to observe that the deletion of $e$ corresponds to
impose the equation $w(e)=0\in H$. Then the result is clear from
the definitions.
\end{proof}

We want to prove that our polynomial
$\widetilde{{\chi}}_{\mathcal{G},\ell}^*(q)$ satisfies the same
recursion.
\begin{Lem}
Let $(\mathcal{G},\ell)$ be a labelled graph and let $q$ be an
admissible integer. For a regular edge $e\in R$ we have
$$
\widetilde{\chi}_{\mathcal{G,\ell}}^*(q)=\widetilde{\chi}_{\mathcal{G}/e,\ell}^*(q)
-\widetilde{\chi}_{\mathcal{G}- e,\ell}^*(q).
$$
\end{Lem}
\begin{proof}
We distinguish three cases.

\underline{Case 1}: $e$ is a proper edge, i.e. the corresponding
edge in $\overline{\mathcal{G}}$ is not a loop and it is contained
in a circuit. Then, applying Theorem \ref{thm:properties} (1), we
have
\begin{eqnarray*}
\widetilde{\chi}_{\mathcal{G,\ell}}^*(q) & = &
(-1)^{|R|-|\overline{V}|+k}q^{|D|-|V|+|\overline{V}|}M_{X_R}(0,1-q)\\
 & = & (-1)^{|R|-|\overline{V}|+k}q^{|D|-|V|+|\overline{V}|}(M_{X_{R_2}}(0,1-q)+M_{X_{R_1}}(0,1-q))\\
 & = &
 (-1)^{|R_2|-|\overline{V}_2|+k}q^{|D_2|-|V|+|\overline{V}_2|} M_{X_{R_2}}(0,1-q)\\
 & - & (-1)^{|R_1|-|\overline{V}_1|+k} q^{|D_1|-|V|+|\overline{V}_1|}M_{X_{R_1}}(0,1-q)\\
 & = & \widetilde{\chi}_{\mathcal{G}/e,\ell}^*(q)-\widetilde{\chi}_{\mathcal{G}- e,\ell}^*(q),
\end{eqnarray*}
since $|R_1|=|R_2|=|R|-1$, $|D_1|=|D|$, $|D_2|=|D|+1$,
$|\overline{V}_1|=|\overline{V}|$ and
$|\overline{V}_2|=|\overline{V}|-1$.

\underline{Case 2}: $e$ is a free edge, i.e. the corresponding
edge in $\overline{\mathcal{G}}$ is not contained in a circuit and
is not a loop. Then, applying Theorem \ref{thm:properties} (2), we
have
\begin{eqnarray*}
\widetilde{\chi}_{\mathcal{G,\ell}}^*(q) & = &
(-1)^{|R|-|\overline{V}|+k}q^{|D|-|V|+|\overline{V}|}M_{X_R}(0,1-q)\\
 & = & (-1)^{|R|-|\overline{V}|+k}q^{|D|-|V|+|\overline{V}|}(M_{X_{R_2}}(0,1-q)-M_{X_{R_1}}(0,1-q))\\
 & = &
 (-1)^{|R_2|-|\overline{V}_2|+k}q^{|D_2|-|V|+|\overline{V}_2|}M_{X_{R_2}}(0,1-q)\\
 & - & (-1)^{|R_1|-|\overline{V}_1|+(k+1)}q^{|D_1|-|V|+|\overline{V}_1|}M_{X_{R_1}}(0,1-q)\\
 & = & \widetilde{\chi}_{\mathcal{G}/e,\ell}^*(q)
-\widetilde{\chi}_{\mathcal{G}- e,\ell}^*(q),
\end{eqnarray*}
since $\mathcal{G}- e$ has now an extra connected component,
$|D_1|=|D|$, $|D_2|=|D|+1$, $|R_1|=|R_2|=|R|-1$,
$|\overline{V}_1|=|\overline{V}|$ and
$|\overline{V}_2|=|\overline{V}|-1$.

\underline{Case 3}: $e$ is a torsion edge, i.e. the corresponding
edge in $\overline{\mathcal{G}}$ is a loop. Then, applying Theorem
\ref{thm:properties} (3), we have
\begin{eqnarray*}
\widetilde{\chi}_{\mathcal{G,\ell}}^*(q) & = &
(-1)^{|R|-|\overline{V}|+k}q^{|D|-|V|+|\overline{V}|}M_{X_R}(0,1-q)\\
 & = & (-1)^{|R|-|\overline{V}|+k}q^{|D|-|V|+|\overline{V}|}(-qM_{X_{R_2}}(0,1-q)+M_{X_{R_1}}(0,1-q))\\
 & = & (-1)^{|R_2|-|\overline{V}_2|+k}q^{|D_2|-|V|+|\overline{V}_2|}M_{X_{R_2}}(0,1-q)\\
 & - & (-1)^{|R_1|-|\overline{V}_1|+k}q^{|D_1|-|V|+|\overline{V}_1|}M_{X_{R_1}}(0,1-q)\\
 & = & \widetilde{\chi}_{\mathcal{G}/e,\ell}^*(q)
-\widetilde{\chi}_{\mathcal{G}- e,\ell}^*(q),
\end{eqnarray*}
since $|D_1|=|D|$, $|D_2|=|D|+1$, $|R_1|=|R_2|=|R|-1$,
$|\overline{V}_1|=|\overline{V}|$ and
$|\overline{V}_2|=|\overline{V}|$.
\end{proof}

In this way we reduce the proof of Theorem \ref{thm:mainflow} to
the case where there are no regular edges. For this case, first of
all we reduce ourself to the case of a connected graph: suppose
that our graph $\mathcal{G}$ has $k$ connected components
$\mathcal{G}^{(1)},\mathcal{G}^{(2)},\dots,\mathcal{G}^{(k)}$ with
the corresponding labellings
$\ell^{(1)},\ell^{(2)},\dots,\ell^{(k)}$. In this case the matroid
$\mathfrak{M}_{\mathcal{G},\ell}$ is the direct sum of the
matroids $\mathfrak{M}_{\mathcal{G}^{(1)},\ell^{(1)}},
\mathfrak{M}_{\mathcal{G}^{(2)},\ell^{(2)}},\dots,
\mathfrak{M}_{\mathcal{G}^{(k)},\ell^{(k)}}$ (cf. Remark
\ref{rm:directsum}). Since $\overline{\mathcal{G}^{(i)}}$ consists
of a single vertex with no edges for $i=1,2,\dots,k$, we have
$|\overline{V}|=k$. We denote by $D^{(i)}$ and $V^{(i)}$ the set
of dotted edges and vertices respectively of $\mathcal{G}^{(i)}$.
Therefore, assuming the result for a connected graph, we have
\begin{eqnarray*}
(-1)^{|R|-|\overline{V}|+k}q^{|D|-|V|+|\overline{V}|}M_{\mathcal{G},\ell}(0,1-q)
& = &
q^{|D|-|V|+k}M_{\mathcal{G},\ell}(0,1-q)\\
\text{(by Theorem \ref{thm:properties} (4))} & = &
q^{|D|-|V|+k}\cdot \prod_{i=1}^k
M_{\mathcal{G}^{(i)},\ell^{(i)}}(0,1-q)\\
 & = & \prod_{i=1}^k q^{|D^{(i)}|-|V^{(i)}|+1}M_{\mathcal{G}^{(i)},\ell^{(i)}}(0,1-q)\\
\text{(by assumption on connected graphs)} & = & \prod_{i=1}^k
{\chi}_{\mathcal{G}^{(i)},\ell^{(i)}}^*(q)=
{\chi}_{\mathcal{G},\ell}^*(q),
\end{eqnarray*}
where the last equality is clear from the definition of arithmetic
flow polynomial.

So we are left to prove the following lemma.
\begin{Lem}
Let $(\mathcal{G},\ell)$ be a labelled connected graph with no
regular edges and let $q$ be an admissible integer. We have
$$
{\chi}_{\mathcal{G},\ell}^*(q)=\widetilde{{\chi}}_{\mathcal{G},\ell}^*(q).
$$
\end{Lem}
\begin{proof}
First of all, if we set $m:=|E|=|D|$ and $n:=|V|$, notice that
$$
\widetilde{{\chi}}_{\mathcal{G},\ell}^*(q)=q^{m-n+1}\cdot
M_{\mathcal{G},\ell}(0,1-q) =q^{m-n+1}\cdot m(\emptyset),
$$
where $m(\emptyset)$ is the cardinality of the torsion subgroup of
the group $\mathbb{Z}^n/\langle X_D\rangle$, i.e. the GCD of the
maximal rank nonzero minors of the matrix $[X_D]$ whose columns
are the elements of $X_D$ (cf. Remark \ref{rm:GCD}).

To compute the number of arithmetic $q$-flows, observe that
$m=|X_D|$, and $m\geq n-1$ (we are assuming that $\mathcal{G}$ is
connected). We can look at $[X_D]$ as a linear operator (acting on
the left)
$$
[X_D]:(\mathbb{Z}/q\mathbb{Z})^m\rightarrow
(\mathbb{Z}/q\mathbb{Z})^n.
$$
Then the elements of the kernel correspond naturally to the
arithmetic $q$-flows, hence we need to compute the cardinality of
$Ker [X_D]$. Looking at the exact sequence (i.e. the image of each
homomorphism is the kernel of the successive one)
$$
0 \rightarrow Ker [X_D]\mathop{\hookrightarrow}^{\iota}
(\mathbb{Z}/q\mathbb{Z})^m \mathop{\longrightarrow}^{[X_D]}
(\mathbb{Z}/q\mathbb{Z})^n \mathop{\longrightarrow}^{\pi}
\frac{(\mathbb{Z}/q\mathbb{Z})^n}{Im [X_D]}\rightarrow 0,
$$
where $\iota$ and $\pi$ are the inclusion and the natural map to
the quotient respectively, we realize that
$(\mathbb{Z}/q\mathbb{Z})^m/Ker [X_D]\cong Im [X_D]$ implies
$$
\frac{|(\mathbb{Z}/q\mathbb{Z})^n|}{|Im
[X_D]|}=\frac{|(\mathbb{Z}/q\mathbb{Z})^n|}{|(\mathbb{Z}/q\mathbb{Z})^m|}\cdot
|Ker [X_D]|= q^{n-m} \cdot |Ker [X_D]|.
$$
Now by the isomorphism theorem
$$
|(\mathbb{Z}/q\mathbb{Z})^n:Im [X_D]|=|\mathbb{Z}^n:Im [X_D\cup
Q]|,
$$
where $Q=\{q_1,q_2,\dots,q_n\}$, $q_i$ is the vector of
$\mathbb{Z}^n$ with $1$ in the $i$-th position and $0$ elsewhere,
and we see the matrix $[X_D\cup Q]$ (whose columns are the
elements of $X_D$ and of $Q$) as a linear operator (acting on the
left)
$$
[X_D\cup Q]:\mathbb{Z}^{n+m}\rightarrow \mathbb{Z}^{n}.
$$
But when we compute the right-hand side of the last equation we
get $q\cdot m(\emptyset)$. In fact by Remark \ref{rm:GCD} this the
$GCD$ of the minors of maximal rank of $[X_D\cup Q]$. But, since
$\ell(e)$ divides $q$ for every edge $e\in E$ ($=D$ in this case),
it is enough to compute the minors of maximal rank which involve
the edges of a spanning tree plus an extra columns from $Q$ (the
other ones are multiples of these). Therefore we get $q$ times the
$GCD$ of minors of rank $n-1$ in $[X_D]$, which are the ones
corresponding to spanning trees. But this last number is exactly
$m(\emptyset)$.

From all this we get
$$
\chi_{\mathcal{G},\ell}^*(q)=|Ker([X_D])|=q^{m-n+1}\cdot
m(\emptyset)=q^{m-n+1}\cdot
M_{\mathcal{G},\ell}(0,1-q)=\widetilde{\chi}_{\mathcal{G},\ell}^*(q),
$$
as we wanted.
\end{proof}

\section{Final comments}

Recall that the \textit{dual} of a matroid $\mathfrak{M}_X=(X,rk)$
is the matroid $\mathfrak{M}_{X}^*=(X,rk^*)$ on $X$ whose rank
function $rk^*$ is given by the formula
$$
rk^*(A):=|X|-rk(X)+rk(X\setminus A)
$$
for all $A\subseteq X$.

If we denote by $T(\mathfrak{M}_X;x,y)$ and
$T(\mathfrak{M}_{X}^*;x,y)$ the Tutte polynomial of
$\mathfrak{M}_X$ and $\mathfrak{M}_{X}^*$ respectively, then we
have the easy relation
$$
T(\mathfrak{M}_{X}^*;x,y)=T(\mathfrak{M}_X;y,x).
$$

Consider a graph $\mathcal{G}$ whose associated matroid is
$\mathfrak{M}_{\mathcal{G}}$ and suppose now that the dual matroid
$\mathfrak{M}_{\mathcal{G}}^*$ is realized by another graph
$\mathcal{G}^*$. In this case $\mathcal{G}^*$ is called a
\textit{dual graph} of $\mathcal{G}$.

It turns out that the graphs $\mathcal{G}$ that admit such a dual
are exactly the \textit{planar graphs}, for which a dual graph is
still planar and it can be constructed explicitly starting from
the graph and its planar embedding (see \cite[Section 1.8]{GR}).

These remarks together with Corollaries \ref{cor:color} and
\ref{cor:flow} suggest that there should be a sort of duality
between the colorings of $\mathcal{G}$ and the flows of
$\mathcal{G}^*$.

Similarly, the \textit{dual} of an arithmetic matroid
$(\mathfrak{M}_X,m)$ is simply the arithmetic matroid
$(\mathfrak{M}_{X}^*,m^*)$, where $m^*(A):=m(X\setminus A)$ for
all $A\subseteq X$ (see \cite{DM2} for the proofs of the
statements in this discussion).

If we denote by $M(\mathfrak{M}_X;x,y)$ and
$M(\mathfrak{M}_{X}^*;x,y)$ the arithmetic Tutte polynomial of
$(\mathfrak{M}_X,m)$ and $(\mathfrak{M}_{X}^*,m^*)$ respectively,
then we have the easy relation
$$
M(\mathfrak{M}_{X}^*;x,y)=M(\mathfrak{M}_X;y,x).
$$

Consider a labelled graph $(\mathcal{G},\ell)$ whose associated
arithmetic matroid is $\mathfrak{M}_{\mathcal{G},\ell}$ and
suppose now that the dual matroid
$\mathfrak{M}_{\mathcal{G},\ell}^*$ is realized by another
labelled graph $(\mathcal{G}^*,\ell^*)$, which we call a
\textit{dual} of $(\mathcal{G},\ell)$.

These remarks together with Theorems \ref{thm:maincolor} and
\ref{thm:mainflow} suggest that there should be a sort of duality
between the arithmetic colorings of $(\mathcal{G},\ell)$ and the
arithmetic flows of $(\mathcal{G}^*,\ell^*)$.

It is now natural to formulate the following vague problem, that
we leave open.

\begin{Pb}
Find a topological-arithmetic characterization (and eventually an
explicit construction) of the dual of a labelled graph.
\end{Pb}
We finally remark that for passing from the labelled graph to its
arithmetic matroid we went through a list of elements in a
finitely generated abelian group. For these in \cite{DM2} we
proved (giving an explicit construction) that the dual matroid
arise again from such a list.

\end{document}